\documentclass[11pt,twoside]{amsart}
\usepackage{latexsym,amssymb,amsmath}
\usepackage[all]{xy}
\usepackage{mathrsfs}
\usepackage{float}
\usepackage{hyperref}
\usepackage{tikz}
\usepackage{multicol}

\numberwithin{equation}{section}
\newtheorem{theorem}{Theorem}[section]
\newtheorem{lemma}[theorem]{Lemma}
\newtheorem{proposition}[theorem]{Proposition}
\newtheorem{corollary}[theorem]{Corollary}
\newtheorem{conjecture}[theorem]{Conjecture}

\theoremstyle{definition}
\newtheorem{definition}[theorem]{Definition}
\newtheorem{procedure}[theorem]{Procedure}
\newtheorem{remark}[theorem]{Remark}
\newtheorem{example}[theorem]{Example}

\begin{document}

\title[Systems of parameters for monomial ideal quotients]{Systems of
parameters consisting of linear 
forms for monomial ideal quotients}

\thanks{The first author was supported by a scholarship from Secihti.
The other two authors were supported by SNII}

\author[E. A. Contreras]{Edwin A. Contreras}
\address{
Departamento de Matem\'aticas\\
Cinvestav, Av. IPN 2508, 07360, CDMX, M\'exico.
}
\email{acontreras@math.cinvestav.mx}

\author[E. Reyes]{Enrique Reyes}
\address{
Departamento de Matem\'aticas\\
Cinvestav, Av. IPN 2508, 07360, CDMX, M\'exico.
}
\email{ereyes@math.cinvestav.mx}

\author[R. H. Villarreal]{Rafael H. Villarreal}
\address{
Departamento de Matem\'aticas\\
Cinvestav, Av. IPN 2508, 07360, CDMX, M\'exico.
}
\email{rvillarreal@cinvestav.mx}



\begin{abstract} 
Let $S=K[x_1,\ldots,x_n]$ be a polynomial ring over a field $K$ and
let $I$ be a monomial ideal of $S$. We classify linear systems of
parameters of $S/I$ over any field $K$ using linear algebra and show
explicit linear systems of parameters when $K$ has at least $n$
elements. If $I(G)$ is the edge ideal of a perfect graph $G$, a cycle
or the complement of a cycle, we show
that $S/I(G)$ has a 0-1 linear system of parameters, and for graphs with independence number 
equal to $2$, we characterize when $S/I(G)$ has a 0-1 linear system of
parameters. 
\end{abstract}

\maketitle

\section{Introduction}

Systems of parameters of standard graded algebras over fields always
exist (Theorem~\ref{hsop}) and they play an
important role in combinatorial commutative algebra, specially in the
study of Cohen--Macaulay algebras and Hilbert series
\cite{BHer,Eisen,readdy,Sta1,Sta2,monalg-3rd-edition}. They can be
interpreted in terms of Noether normalizations of these algebras
(Proposition~\ref{sop-noether}).

Let $S=K[x_1,\ldots,x_n]=\bigoplus_{i=0}^\infty S_i$ be a polynomial
ring over a field $K$ with the standard grading and
let $I$ be a monomial ideal of $S$. The radical of $I$ is denoted by
${\rm rad}(I)$. A set of homogeneous
elements $\underline{f}=\{f_1,\ldots,f_d\}$ in $S$  
is called a {\it system of
parameters\/} (s.o.p. for
short) for $S/I$ if $d=\dim(S/I)$ and 
$$
{\rm rad}(I,f_1,\ldots,f_d)=\mathfrak{m}, 
$$
where $\mathfrak{m}=(x_1, \ldots,x_n)$ is the maximal ideal of $S$. This is equivalent to the
usual definition of a system of parameters
(Definition~\ref{sop-def}). By a result of Stanley, $S/I$ 
has a non-linear Noether normalization given by the elementary
symmetric polynomials up to degree $d$
(Proposition~\ref{stanley-sop}). 

In this work, we focus in the study of
linear systems of parameters and particularly in those where
the linear forms $f_1,\ldots,f_d$ have 0-1 coefficients, that is,
coefficients in $\{0,1\}$.   
If $K$ is an infinite field, a system of
parameters for $S/I$ consisting of linear forms always exist (Theorem~\ref{hsop}) and can be
computed using \textit{Macaulay}$2$ \cite{mac2} and the
\textit{NoetherNormalization} package
\cite{noethernormalization}. 

For Stanley--Reisner rings and for quotients of monomial ideals of 
K\"onig type, Herzog and Moradi exhibit non-linear universal systems
of parameters \cite[Theorem~3.1]{2021-HerzogMoradi}
and linear systems of parameters
 \cite[Theorem~2.3]{2021-HerzogMoradi} that are independent
of the field $K$, respectively. For other results on universal
systems of parameters of Stanley--Reisner rings including its depth 
sensitivity, see the papers of Adams and Reiner 
\cite[p.~152, Theorem~5.3]{Adams-Reiner}, Smith \cite{Smith}, and
references therein. 

For convenience, we use the following
multi-index notation to denote the monomials of $S$:
\[
x^{a}:=x_1^{a_{1}}\cdots x_n^{a_{n}}\ \text{ for } a=(a_{1},\ldots,a_n)\in \mathbb{N}^n.
\]
\quad A prime ideal $\mathfrak{p}$ of $S$ is an \textit{associated
prime} of $S/I$ 
if $\mathfrak{p}=(I\colon x^a)$ for some monomial $x^a\in S$. It is
not hard to see that any associated prime of $S/I$ is generated by a
subset of $\{x_1,\ldots,x_n\}$. 
The set of associated primes of $S/I$ is denoted by ${\rm Ass}(I)$. The \textit{minimal 
primes} of $I$ are the minimal primes of ${\rm Ass}(I)$ with respect to
inclusion.

We come to one of our main results that gives a \textit{rank
criterion} for linear systems of parameters for monomial ideals
quotients, 
and that is also valid for graded
ideals whose minimal primes are monomial ideals (Remark~\ref{sop-min-mono}).

\noindent {\bf Theorem~\ref{rank-criterion}.}\textit{
Let $S=K[x_1,\ldots,x_n]$ be a polynomial ring over a field $K$, let
$I$ be a monomial ideal of $S$, and let $d= \dim(S/I)$. A set $\{f_1,\ldots,f_d\}$,
$f_i=\sum_{j=1}^na_{i,j}x_j$, $a_{i,j}\in K$, $i=1,\ldots,d$, is a system of
parameters for $S/I$ if and only if for each minimal prime
$\mathfrak{p}$ of $I$ the submatrix 
$$
A_\mathfrak{p}:=(a_{i,j})_{x_j\notin\mathfrak{p}}
$$
of the $d\times n$ matrix $A:=(a_{i,j})$ has full column rank. 
}

As an application, the quotient ring of a monomial ideal $I$ and the
quotient ring of its radical have the same
linear systems of parameters, because $I$ and its radical have the same minimal primes
and the same dimension (Corollary~\ref{coro-I-radI}). We obtain 
a Cohen--Macaulay classification of monomial
ideals of $S$ of dimension $2$ over a field $K$ of characteristic $0$
using a universal explicit linear system of parameters that works for
all monomial ideals of dimension $2$ (Corollary~\ref{coro-dim=2}).

To give a version of the rank criterion using the combinatorics of $I$, we
now introduce clutters and their edge ideals and invariants. 
A {\it clutter\/} $\mathcal{C}$
with vertex set $V(\mathcal{C})=\{x_1,\ldots,x_n\}$ is a family $E(\mathcal{C})$ of subsets of 
$V(\mathcal{C})$ called edges none of which is included in
another. One example of a clutter is a graph $G$ with the
vertices and edges defined in the usual way. The {\it edge
ideal\/} of a clutter $\mathcal{C}$, denoted by $I(\mathcal{C})$,
is the ideal of $S$ given by
$$
I(\mathcal{C}):=(\{\textstyle\prod_{x_i\in e}x_i \mid e\in E(\mathcal{C})\}).
$$
\quad Edge ideals of graphs and clutters were 
introduced in \cite{Vi2} and \cite{clutters,Ha-VanTuyl}, 
respectively. 
A subset $F$ of $V(\mathcal{C})$ is called 
{\it independent\/} or {\it
stable\/} if $e\not\subset F$ for any  
$e\in E(\mathcal{C})$, and a subset of $V(\mathcal{C})$ is a \textit{vertex
cover} if and only if its complement is an independent set. The number of
vertices in any smallest vertex cover, denoted by 
$\alpha_0({\mathcal C})$, is the {\it covering
number}. The {\it independence
number}, denoted by
$\beta_0(\mathcal{C})$, is the number 
of vertices in 
any largest independent set. The algebraic translation is that the
height of $I(\mathcal{C})$ is $\alpha_0({\mathcal C})$ and the
dimension of $S/I(\mathcal{C})$ is $\beta_0(\mathcal{C})$
\cite[p.~199]{monalg-3rd-edition}. 

The independent sets of a clutter $\mathcal{C}$ are the faces of the
independence complex $\Delta_\mathcal{C}$ of the Stanley--Reisner ring
$S/I(\mathcal{C})$. We recover a result of Stanley that was stated in
\cite[p.~150]{Stanley-sop} using the maximal faces of
$\Delta_\mathcal{C}$. Kind and Kleinschmidt show this result when
$\Delta_\mathcal{C}$ is a 
pure shellable complex \cite[Satz, p.~175]{kind}, see also
\cite[Theorem~2.1]{readdy}.

\noindent {\bf Corollary~\ref{sop-clutters}.} \cite[p.~150]{Stanley-sop}\textit{
Let $I=I(\mathcal{C})$ be the edge ideal of a clutter $\mathcal{C}$, and
let $d=\beta_0(\mathcal{C})$ be the dimension of $S/I$. A set $\{f_1,\ldots,f_d\}$,
$f_i=\sum_{j=1}^na_{i,j}x_j$, $a_{i,j}\in K$, $i=1,\ldots,d$, is a system of
parameters for $S/I$ if and only if for each maximal independent set 
$F$ of $\mathcal{C}$ the $d\times|F|$ submatrix 
$$
A_F:=(a_{i,j})_{x_j\in F}
$$
of the $d\times n$ matrix $A:=(a_{i,j})$ has full column rank 
equal to $|F|$. 
}

If the
field $K$ has at least $n$ elements, we give an explicit linear
system of parameters for $S/I$  
 that can
be used to determine when $S/I$ is Cohen--Macaulay, and to
compute the multiplicity, the canonical module, and the type of these
rings \cite{BHer,2021-HerzogMoradi,monalg-3rd-edition}
(Theorem~\ref{explicit-lsop}, Example~\ref{explicit-lsop-ex}). 
We can use generator matrices of \textit{maximum distance separable
linear codes} (MDS codes) to construct linear
systems of parameters of monomial ideals over finite fields 
(Proposition~\ref{mds-sop}, Example~\ref{sop-from-ff}). This is
possible due to the fact that MDS codes can be characterized by the 
maximal minors of generator matrices (Theorem~\ref{mds-minors}).   

We study 0-1 linear systems of 
parameters of quotient rings of monomial ideals
(Definition~\ref{0-1-sop-def}), where there are already some results
in the literature \cite{2021-HerzogMoradi}. In Section~\ref{section-graphs}, we
focus on those rings arising from edge ideals of graphs. 

Let $G$ be a graph and let $I(G)$ be its edge ideal. If
$\beta_0(G)=2$, we prove that $S/I(G)$ has a  0-1
linear system 
of parameters if and only if $V(G)$ has a partition in three or in
two cliques of $G$ (Definition~\ref{clique},
Proposition~\ref{beta_0=2}). 
We show that the complement of the Mycielski--Gr\"otzsch graph has no 
0-1 linear system of parameters over any field $K$ 
(Proposition~\ref{no-0-1-sop-prop}, Example~\ref{example-comp-Grotzsch-graph}).

 The edge ring
$S/I(G)$ of $G$ has a 0-1 linear system of parameters in the following cases.
\begin{enumerate}
\item[\rm(a)] Perfect graphs (Definition~\ref{perfect-def},
Theorem~\ref{perfect-sop});
\item[\rm(b)] K\"onig graphs
(Definition~\ref{konig-def}, Proposition~\ref{konig-0-1-sop});
\item[\rm(c)] A cycle of length at least $4$ and its complement 
(Propositions~\ref{konig-0-1-sop}, \ref{odd-cycle}, and
\ref{antihole-star-sop});
\item[\rm(d)] $\alpha_0$-reducible graphs whose components have 0-1 edge rings
with linear systems of parameters (Definition~\ref{reducible-def},
Proposition~\ref{reduction-star-0-1}).  
\end{enumerate}

Item (b) holds more generally for K\"onig monomial ideals 
\cite{2021-HerzogMoradi}. 
We introduce the notion of a star system of parameters
of $S/I(G)$ (Definition~\ref{star-sop}). From our results all graphs
in items (b) and (c) have star systems of parameters. If $G$ is
an $\alpha_0$-reducible graph whose components have 0-1 edge rings 
with star systems of parameters, then also $S/I(G)$ has a star system of
parameters. If 
$S/I(G)$ has a 0-1 system of parameters, we conjecture that 
$S/I(G)$ also has a star-system of parameters
associated to some maximal independent set $\{z_1, \ldots, z_d\}$ of
$G$ (Conjecture~\ref{conjecture-star}). 

In Section~\ref{examples-section}, we give several examples illustrating some of our 
results and, in Appendix~\ref{procedures-sop}, we give the procedures for
\textit{Macaulay}$2$ \cite{mac2} that we use to compute linear systems of
parameters of monomial ideals quotients and their algebraic and
combinatorial invariants. In
Procedure~\ref{0-1-sop}, we define a function zeroOneSOP that tests for the 
existence of a 0-1 linear system of parameters for a monomial ideal
quotient and returns one when it exists.

For unexplained terminology and additional information, we refer to
\cite{BHer,Mats,Sta2} for commutative algebra and
\cite{herzog-hibi-book,Vi2,monalg-3rd-edition} for monomial ideals and
edge ideals of graphs.

\section{Preliminaries}\label{prelim-section}

\begin{definition}\label{sop-def}\rm
Let $K$ be a field and let $R=\bigoplus_{i=0}^\infty R_i$ be a
standard graded $K$-algebra. A set of homogeneous
elements $\{f_1,\ldots,f_d\}$ of $R$ 
is called a {\it system of 
parameters\/} for $R$ if $d=\dim(R)$ and ${\rm
rad}(f_1,\ldots,f_d)=R_+=\bigoplus_{i=1}^\infty R_i$. If $R$ is
presented as $S/I$, where
$S=K[x_1,\ldots,x_n]=\bigoplus_{i=0}^\infty S_i$ is 
a polynomial ring with the standard grading and 
$I$ is a graded ideal of $S$, we say that a set of homogeneous
polynomials $\underline{f}=\{f_1,\ldots,f_d\}\subset S$ 
is a {\it system of
parameters\/} for $S/I$ if the image of $\underline{f}$ under the
canonical map $\varphi\colon S\rightarrow S/I$, $h\mapsto h+I$, is a system of parameters for $S/I$. 
\end{definition}



\begin{theorem}\cite[Theorem 3.1.23]{monalg-3rd-edition}\label{hsop}
Let $S$ be a polynomial ring over a 
field $K$ and let $I$ be a graded ideal of $S$. Then, there are
homogeneous polynomials $f_1,\ldots,f_d$ in $S$ such that 
\[
\dim(S/(I,f_1,\ldots,f_i))=d-i,\, \, \mbox{ for }i=1,\ldots,d,
\] 
where $d=\dim(S/I)$. In particular, for $i=d$, ${\rm
rad}(I,f_1,\ldots,f_d)=\mathfrak{m}$. Furthermore, if $K$ is an
infinite field, then $f_1,\ldots,f_d$ can be chosen to be linear
forms. 
\end{theorem}

\begin{theorem}\cite[Theorem 3.1.24]{monalg-3rd-edition}\label{noethern-grad}
Let $S$ be a polynomial ring over an infinite field $K$ and let $I$  
be a graded ideal of $S$. Then, there are 
linear forms $f_1,\ldots,f_d$ in $S$, 
with $d=\dim(S/I)$, and a natural embedding 
\[
A=K[f_1,\ldots,f_d]\stackrel{\varphi}{\hookrightarrow} S/I 
\] 
such that $S/I$ is a finitely generated $A$-module. 
\end{theorem}

\begin{definition}\label{noether-def}\rm
Let $S$ be a polynomial ring over a
field $K$ and $I$ a graded ideal of $S$. A {\it Noether
normalization\/} of 
$S/I$ is
an integral extension 
$K[f_1,\ldots,f_d]\hookrightarrow S/I$, 
where $f_1,\ldots,f_d$ are
homogeneous polynomials in $S$ and $d=\dim(S/I)$.
\end{definition}

Noether normalizations and systems of parameters are equivalent
notions. 

\begin{proposition}\label{sop-noether} Let $S=K[x_1,\ldots,x_n]$ be a
polynomial ring over a field $K$ and
let $I$  be a graded ideal of $S$. A set of homogeneous polynomials
$\underline{f}=\{f_1,\ldots,f_d\}$ is a system of parameters for $S/I$
if and only if the canonical map 
\[
A=K[f_1,\ldots,f_d]\stackrel{\varphi}{\longrightarrow} S/I 
\] 
is an embedding and $S/I$ is a finitely generated $A$-module. 
\end{proposition}

\begin{proof} 
$\Rightarrow$) Assume that $\underline{f}$ is a system of
parameters. Then, ${\rm rad}(I,f_1,\ldots,f_d)=\mathfrak{m}$,
$d=\dim(S/I)$, and this
implication follows from the proof of
\cite[Theorem~3.1.24]{monalg-3rd-edition}. 

$\Leftarrow$) Assume that the embedding $A\hookrightarrow S/I$ 
is a finite extension and $d=\dim(S/I)$. To show the equality 
${\rm rad}(I,f_1,\ldots,f_d)=\mathfrak{m}$ we need only show the
inclusion ``$\supset$'' because $I$ is graded. Take
$x_i\in\mathfrak{m}$. As $S/I$ is integral over $A$, $\overline{x_i}$
satisfies an equation of the form
$$
a_0+a_1\overline{x_i}+\cdots+a_{n-1}\overline{x_i}^{m-1}+\overline{x_i}^m=\overline{0},
$$
where $a_i\in A$ for all $i$ and $m\geq 1$. Thus, we get
$$a_0+a_1{x_i}+\cdots+a_{n-1}{x_i}^{m-1}
+{x_i}^m\in I,$$
 and since $I$ is graded, the $m$-th homogeneous component of
 this polynomial is in $I$. Therefore, $x_i^m\in(I,f_1,\ldots,f_d)$ and
 $x_i\in {\rm rad}(I,f_1,\ldots,f_d)$.
\end{proof}


\begin{proposition}
{\rm(Stanley, \cite[Proposition~3.1.26]{monalg-3rd-edition})}\label{stanley-sop}
Let $S=K[x_1,\ldots,x_n]$ be a polynomial ring over a field $K$ and 
let $I$ be a monomial ideal with $d=\dim(S/I)$. Then
\[
A=K[\sigma_1,\ldots,\sigma_d]{\hookrightarrow} S/I
\]
is a Noether normalization, where $\sigma_i=\sum_{1 \le j_1 < j_2 <
\dots < j_i \le n} x_{j_1} x_{j_2} \dots x_{j_i}$ is the $i$-{\rm th} 
symmetric polynomial.  
\end{proposition}

For convenience we define regular sequences in a way that can be used
to check this property using \textit{Macaulay}$2$ \cite{mac2}
(Example~\ref{example-comp-Grotzsch-graph}), cf. 
\cite[Definition~2.3.4]{monalg-3rd-edition}.

\begin{definition}\label{regular-def} Let $I\subset S$ be a graded
ideal. A sequence
$\underline{f}=\{f_1,\ldots,f_d\}$ of forms is called a \textit{regular sequence} on
$S/I$ if $((I,f_1,\ldots,f_i)\colon f_{i+1})=(I,f_1,\ldots,f_i)$ for
$i=0,\ldots,d-1$, where we let $f_0=0$ by convention.
\end{definition}

Systems of parameters are closely related to Cohen--Macaulay rings.

\begin{proposition}
\cite[Proposition~3.1.20]{monalg-3rd-edition}\label{cm-iff-sop}
Let $I\subset S$ be a graded ideal and let
$\underline{f}=\{f_1,\ldots,f_d\}$ be a homogeneous system of
parameters for $S/I$. Then,   
$S/I$ is Cohen--Macaulay if and only if $\underline{f}$ is a
regular sequence on $S/I$.
\end{proposition}

\section{Linear algebra criterion for systems of parameters}

The next lemma is well-known and it depends on the field $K$. As is seen below, there is a 
related characterization of left-invertible matrices which is independent of
the field $K$.  

\begin{lemma}\label{pseudo-inverse} Let $K=\mathbb{R}$ be the field of
real numbers and let $B$ be
a matrix of size $d\times r$ with full column rank $r$. Then, 
$B^\top B$ is invertible and the rank of $B^\top B$ is $r$. 
\end{lemma}

\begin{proof} It suffices to show that the linear operator $B^\top B\colon
K^r\rightarrow K^r$ is injective. Take $x\in K^r$ such that $(B^\top
B)x^\top=0^\top$. Then, 
$x(B^TB)x^\top=(Bx^\top)^\top Bx^\top=0$. Using that the product on
the left of $0$ is a sum of squares and the assumption that $K\subset
\mathbb{R}$, one has that $Bx^\top=0$. As 
$B$ has full column rank, its columns are linearly
independent over $K$. Then, noticing that  $Bx^\top$ is a linear combination of
the columns of $B$ and using the equality $Bx^\top=0$, we get that
$x=0$. This proves that $B^\top B$ is injective, and the proof is
complete.
\end{proof}

\begin{definition}\rm An $d\times r$ matrix $B$ with entries in a field $K$ is called
\text{left-invertible} if there is an $r\times d$ matrix $L$ such that
$LB=I_r$. 
\end{definition}

To give a linear algebra criterion for linear systems of parameters of
monomial ideal quotients, we need the following result that works for any field $K$.

\begin{theorem}{\rm(\cite[p.~130]{Gentle},
\cite[Theorem~157]{Santos})}\label{left-inverse}
Let $K$ be a field. A matrix $B\in\mathbb{M}_{d\times r}(K)$ is
left-invertible if and only if ${\rm rank}(B)=r$.    
\end{theorem}

We come to one of our main results.

\begin{theorem}\label{rank-criterion}
Let $S=K[x_1,\ldots,x_n]$ be a polynomial ring over a field $K$, let
$I$ be a monomial ideal of $S$, and let $d= \dim(S/I)$. A set $\{f_1,\ldots,f_d\}$,
$f_i=\sum_{j=1}^na_{i,j}x_j$, $a_{i,j}\in K$, $i=1,\ldots,d$, is a system of
parameters for $S/I$ if and only if for each minimal prime
$\mathfrak{p}$ of $I$ the submatrix 
$$
A_\mathfrak{p}:=(a_{i,j})_{x_j\notin\mathfrak{p}}
$$
of size $d\times(n-{\rm ht}(\mathfrak{p}))$ of the $d\times n$ matrix
$A:=(a_{i,j})$ 
has full column rank. 
\end{theorem}

\begin{proof} $\Rightarrow$) As $f_1,\ldots,f_d$ is a system of
parameters for $S/I$, one has 
$$
{\rm rad}(I,f_1,\ldots,f_d)=\mathfrak{m}.
$$
\quad Let $\mathfrak{p}$ be any minimal prime of $I$. We let 
$\overline{f_i}=\sum_{x_j\notin\mathfrak{p}}a_{i,j}x_j$ for 
$i=1,\ldots,d$. Then,
$$
(I,f_1,\ldots,f_d)\subset
(\mathfrak{p},f_1,\ldots,f_d)=(\mathfrak{p},\overline{f_1},\ldots,\overline{f_d})
\subset \mathfrak{m},
$$
and, noticing that $(\mathfrak{p},f_1,\ldots,f_d)$ is a prime ideal, 
taking radicals one has the equality
$$ 
(\mathfrak{p},\overline{f_1},\ldots,\overline{f_d})=\mathfrak{m}.
$$
\quad Then, for each $x_j\notin\mathfrak{p}$ we can write
$x_j=\sum_{i=1}^d b_{i,j}\overline{f_i}$, $b_{i,j}\in K$. Then, 
letting $B=(b_{i,j})$,
$\overline{x}=\sum_{x_j\notin\mathfrak{p}}x_je_j$,  and 
$\overline{f}=(\overline{f_1},\ldots,\overline{f_d})$,  we have
$\overline{x}^\top=B^\top\, \overline{f}^\top$
and since $\overline{f}^\top=A_\mathfrak{p}\overline{x}^\top$, we get
$$
\overline{x}^\top=B^\top\,\overline{f}^\top=B^\top(A_\mathfrak{p}\overline{x}^\top)
=(B^{\top} A_\mathfrak{p})\overline{x}^\top.
$$
\quad Hence, $B^{\top}A_\mathfrak{p}=I_{n-{\rm ht}(\mathfrak{p})}$ is the
identity matrix, that is, $A_\mathfrak{p}$ has left inverse. Therefore, 
by Theorem~\ref{left-inverse}, $A_\mathfrak{p}$ has full column rank. 

$\Leftarrow$) Let $P$ be any prime ideal of $S$ containing
$(I,f_1,\ldots,f_d)$. Then, there is a minimal prime ideal
$\mathfrak{p}$ of $I$ such that $\mathfrak{p}\subset P$ and
consequently 
$$
(I,f_1,\ldots,f_d)\subset (\mathfrak{p},f_1,\ldots,f_d)\subset P.
$$
\quad Letting $Q:=(\mathfrak{p},f_1,\ldots,f_d)$ it suffices to show
that $Q=\mathfrak{m}=(x_1,\ldots,x_n)$ because this will show that
$P=\mathfrak{m}$ and consequently the radical of $(I,f_1,\ldots,f_d)$
is $\mathfrak{m}$. We let 
$$ 
x=(x_1,\ldots,x_n),\ 
\overline{x}=\sum_{x_j\notin\mathfrak{p}}e_jx_j,\ f=(f_1,\ldots,f_d),\
\overline{f_i}=\sum_{x_j\notin\mathfrak{p}}a_{i,j}x_j,\ 
\overline{f}=(\overline{f_1},\ldots,\overline{f_d}),
$$ 
for $i=1,\ldots,d$. Note that
$Q=(\mathfrak{p},\overline{f_1},\ldots,\overline{f_d})$. From the
equality $Ax^\top=f^\top$, we get 
\begin{equation}\label{aug26-26}
A_\mathfrak{p}\overline{x}^\top=\overline{f}^\top.
\end{equation}
\quad Let $r=|\{x_j\mid x_j\notin\mathfrak{p}\}|$ be the number of 
columns of $A_\mathfrak{p}$. Since $A_\mathfrak{p}$ has full column rank, 
by
Theorem~\ref{left-inverse}, $A_\mathfrak{p}$ has a left-inverse, that
is, there is a $r\times d$ matrix $L$ such that $LA_\mathfrak{p}=I_r$
and consequently, multiplying Eq.~\eqref{aug26-26} by $L$, one has   
$$
\overline{x}^\top=I_r\overline{x}^\top=L\overline{f}^\top.
$$
\quad Therefore, $x_i\in(\overline{f_1},\ldots,\overline{f_d})$ for
any $x_i\notin\mathfrak{p}$. Recall that $\mathfrak{p}$ is generated
by a subset of $\{x_i\}_{i=1}^n$ because $I$ is a
monomial ideal. As $Q$ contains all variables in
$\mathfrak{p}$, one has $Q=\mathfrak{m}$. 
\end{proof}

\begin{remark}\label{sop-min-mono} The proof above shows that Theorem~\ref{rank-criterion} 
holds if $I$ is a graded ideal of $S$ whose associated primes are
monomial ideals, that is, they are generated by subsets of variables.
\end{remark}

\begin{corollary}\label{coro-I-radI}
Let $I$ be a monomial ideal of $S$. Then, $\underline{f}$ is a linear
system of parameters for $S/I$ if and only if $\underline{f}$ is a
linear system of parameters for $S/{\rm rad}(I)$.
\end{corollary}

\begin{proof} This follows from the rank criterion
(Theorem~\ref{rank-criterion}) by noticing that $I$ and ${\rm rad}(I)$
are monomial ideals with the same dimension and minimal primes. 
\end{proof}

The following result gives a Cohen--Macaulay criterion for monomial
ideals of dimension $2$ over a field of characteristic $0$. 

\begin{corollary}\label{coro-dim=2}
Let $I$ be a monomial ideal of $S=K[x_1,\ldots,x_n]$ over a field of
${\rm char}(K)=0$. If $d=\dim(S/I)=2$ and $f_1,f_2$ are the linear polynomials
$$
f_1=\sum_{i=1}^nix_i\ \mbox{ and }\ f_2=\sum_{i=2}^n(i-1)x_i,
$$
then $\{f_1,f_2\}$ is a system of parameters for $S/I$. Furthermore, $S/I$ is
a Cohen--Macaulay ring if and only if $\{f_1,f_2\}$ is a regular
sequence of $S/I$.
\end{corollary}

\begin{proof} The coefficients matrix of $\{f_1,f_2\}$ is given by
\[
A=
\begin{pmatrix}
1&2&3&\cdots&i&\cdots&n\\
0&1&2&\cdots&i-1&\cdots&n-1
\end{pmatrix}.
\]
\quad Let $\mathfrak{p}$ be a minimal prime of $I$. Since ${\rm
ht}(I)=n-2$, one has that ${\rm
ht}(\mathfrak{p})=n-2$ or ${\rm
ht}(\mathfrak{p})=n-1$. Then, $A_\mathfrak{p}$ is a $2\times 2$ matrix
or a $2\times 1$ matrix. Thus, by Theorem~\ref{rank-criterion}, 
it suffices to note that 
\[
\det\begin{pmatrix}
i&j\\
i-1&j-1
\end{pmatrix}=j-i.
\]
for all $j\neq i$ and that all columns of $A$ are non-zero. By
Proposition~\ref{cm-iff-sop}, 
$S/I$ is Cohen--Macaulay if and only if $\{f_1,f_2\}$ is a regular
sequence of $S/I$. 
\end{proof}

\begin{corollary}\cite[Remark, p.~150]{Stanley-sop}\label{sop-clutters}
Let $S=K[x_1,\ldots,x_n]$ be a polynomial ring over a field $K$, let
$I=I(\mathcal{C})$ be the edge ideal of a clutter $\mathcal{C}$, and
let $d$ be the dimension of $S/I$. A set $\{f_1,\ldots,f_d\}$,
$f_i=\sum_{j=1}^na_{i,j}x_j$, $a_{i,j}\in K$, $i=1,\ldots,d$, is a system of
parameters for $S/I$ if and only if for each maximal independent set of
vertices $F$ of $\mathcal{C}$ the $d\times|F|$ submatrix 
$$
A_F:=(a_{i,j})_{x_j\in F}
$$
of the $d\times n$ matrix $A:=(a_{i,j})$ has full column rank 
equal to $|F|$. 
\end{corollary}
\begin{proof} The result follows from Theorem~\ref{rank-criterion}, by
noticing a prime ideal $\mathfrak{p}$ of $S$ is a minimal prime of
$I(\mathcal{C})$ if and only if there is a minimal vertex cover $C$ of
$\mathcal{C}$ such that $\mathfrak{p}=(C)$ and 
$F=\{x_i\mid x_i\notin\mathfrak{p}\}$
is a maximal independent set of $\mathcal{C}$.
\end{proof}
The next result gives an explicit linear system of parameters for
$S/I$. 

\begin{theorem}\label{explicit-lsop}
Let $I$ be a monomial ideal of $S=K[x_1,\ldots,x_n]$ over a field $K$
with at least $n$ elements. If $d=\dim(S/I)$ and
$\lambda_1,\ldots,\lambda_n$ are distinct elements of $K$, then the
polynomials 
$$
f_i:=\sum_{j=1}^n\lambda_j^{i-1}x_j\ \mbox{ for }i=1,\ldots,d,
$$
form a linear system of parameters for $S/I$. Furthermore, $S/I$ is
a Cohen--Macaulay ring if and only if $\{f_1,\ldots,f_d\}$ is a 
regular sequence of $S/I$.
\end{theorem}

\begin{proof} The coefficients matrix of $\{f_1,\ldots,f_d\}$ has size
$d\times n$ and is given by
\begin{equation}\label{vandermonde-matrix}
A =(a_{i,j})=
\begin{pmatrix}
1 & 1 & 1 & \cdots & 1 \\
\mathrm{\lambda}_1 & \mathrm{\lambda}_2 & \mathrm{\lambda}_3 & \cdots
& \mathrm{\lambda}_n \\ 
\mathrm{\lambda}_1^2 & \mathrm{\lambda}_2^2 & \mathrm{\lambda}_3^2 &
\cdots & \mathrm{\lambda}_n^2 \\ 
\vdots & \vdots & \vdots & & \vdots \\
\ \ \mathrm{\lambda}_1^{d-1} &\ \  \mathrm{\lambda}_2^{d-1} &
\ \ \mathrm{\lambda}_3^{d-1} & \cdots &\ \  \mathrm{\lambda}_n^{d-1} 
\end{pmatrix}.
\end{equation}
\quad For each choice of $r$ columns of $A$ indexed by
$1\leq i_1<\cdots<i_{r}\leq n$ such that $r\leq d$, the
resulting $d\times r$ matrix is given by
\[
A_{r,d}: =
\begin{pmatrix}
1 & 1 & 1 & \cdots & 1 \\
\mathrm{\lambda}_{i_1} & \mathrm{\lambda}_{i_2} &
\mathrm{\lambda}_{i_3} & \cdots
& \mathrm{\lambda}_{i_r} \\ 
\mathrm{\lambda}_{i_1}^2 & \mathrm{\lambda}_{i_2}^2 &
\mathrm{\lambda}_{i_3}^2 &
\cdots & \mathrm{\lambda}_{i_r}^2 \\ 
\vdots & \vdots & \vdots & & \vdots \\
\ \ \ \mathrm{\lambda}_{i_1}^{d-1} &\ \ \ \mathrm{\lambda}_{i_2}^{d-1} &
\ \ \ \mathrm{\lambda}_{i_3}^{d-1} & \cdots &\ \ \ \mathrm{\lambda}_{i_r}^{d-1} 
\end{pmatrix}.
\]
\quad This matrix has rank $r$, because picking the first $r$ rows of
$A_{r,d}$ gives an $r\times r$ Vandermonde submatrix $A_{r,r}$ and 
\[
\det(A_{r,r}) =
\det\begin{pmatrix}
1 & 1 & 1 & \cdots & 1 \\
\mathrm{\lambda}_{i_1} & \mathrm{\lambda}_{i_2} &
\mathrm{\lambda}_{i_3} & \cdots
& \mathrm{\lambda}_{i_r} \\ 
\mathrm{\lambda}_{i_1}^2 & \mathrm{\lambda}_{i_2}^2 &
\mathrm{\lambda}_{i_3}^2 &
\cdots & \mathrm{\lambda}_{i_r}^2 \\ 
\vdots & \vdots & \vdots & & \vdots \\
\ \ \ \mathrm{\lambda}_{i_1}^{r-1} &\ \ \ \mathrm{\lambda}_{i_2}^{r-1} &
\ \ \ \mathrm{\lambda}_{i_3}^{r-1} & \cdots &\ \ \
\mathrm{\lambda}_{i_r}^{r-1} 
\end{pmatrix}=\prod_{1\leq j<k\leq
r}(\lambda_{i_k}-\lambda_{i_j})\neq 0.
\]
\quad Hence, we have shown that for any $r\leq d$ every set of $r$
columns of $A$  is linearly independent. In particular, the rank of
$A$ is $d$. To prove that $\underline{f}=\{f_1,\ldots,f_d\}$ is a system of
parameters of $S/I$ we now use the rank criterion
(Theorem~\ref{rank-criterion}). Let $\mathfrak{p}$ be a minimal prime of $I$. Since ${\rm
ht}(I)=n-d$, one has that 
${\rm
ht}(\mathfrak{p})\geq n-d$ and $r:=n-{\rm ht}(\mathfrak{p})\leq d$.
Then, since any set of $r\leq d$ columns of $A$ is linearly
independent, the $d\times r$ submatrix 
$$
A_\mathfrak{p}:=(a_{i,j})_{x_j\notin\mathfrak{p}}
$$
of $A:=(a_{i,j})$ has full column rank equal to $r$. Hence,
by the rank criterion, $\underline{f}$ is a system of parameters of
$S/I$. By Proposition~\ref{cm-iff-sop}, the ring 
$S/I$ is Cohen--Macaulay if and only if $\underline{f}$ is a regular
sequence of $S/I$. 
\end{proof}
It is well known that over finite field the coefficient matrix $A$ of
$\underline{f}$ given above in Eq.~\eqref{vandermonde-matrix} is
related to MDS linear codes  (Theorem~\ref{mds-minors}).
For convenience we recall this relation.

Let $K=\mathbb{F}_q$ be a finite field and let $C$ be an $[n,d,\delta]_q$ {\it linear
code} of {\it length}
$n$, {\it dimension} $d$, and \textit{minimum distance} $\delta$, that
is, $C$ is a linear subspace of $K^n$ with $d=\dim_K(C)$, and
$\delta=\delta(C)$ is the smallest number of nonzero entries in any
non-zero vector 
of $C$. The Singleton bound
for the minimum distance of $C$ is 
\cite[Theorem~7.10.6]{Huffman-Pless}:   
$$
1\leq \delta(C)\leq n - d + 1.
$$
\quad If equality holds, $C$ is called a \textit{maximum distance separable
code} (MDS code for short). 
One of the main problems in this area is the MDS
conjecture, see the works
of Ball and De~Beule \cite{Ball1,Ball2,Ball3}, the book of Dougherty
\cite[p.~284]{dougherty}, and references there in. In particular, the
conjecture says that an MDS linear over a prime field 
has length at most $n\leq q+1$ \cite[p.~284]{dougherty}. 

\begin{theorem}{\rm\cite[Corollary~2.2]{Ghorpade-Lachaud}}\label{mds-minors} 
Let $C$ be an $[n,d,\delta]_q$ linear code over $\mathbb{F}_q$ and let $A$ be
a $d\times n$ matrix whose rows form a basis of $C$. Then, $C$ is an
MDS linear code if and only if all
$d\times d$ minors of $A$ are nonzero. 
\end{theorem}



Given an $[n,d,\delta]_q$ linear code $C$ there is an $d\times n$
\textit{generator matrix} whose rows generate $C$ as a vector space
over $\mathbb{F}_q$. One can use
generator matrices of MDS codes to construct linear
systems of parameters of monomial ideals over finite fields. 

\begin{proposition}\label{mds-sop}
Let $A=(a_{i,j})$ be an $d\times n$, $1<d<n$, generator matrix 
of an $[n,d,n-d+1]_q$ linear code $C$ and let 
$f_i=\sum_{j=1}^na_{i,j}x_j$, $a_{i,j}\in \mathbb{F}_q$,
$i=1,\ldots,d$. Then, $\{f_1,\ldots,f_d\}$ is a linear system of parameters of
$S/I$ for any monomial ideal $I$ of
$S=\mathbb{F}_q[x_1,\ldots,x_n]$ of
height $n-d$. 
\end{proposition}

\begin{proof}
Let $I$ be a monomial ideal of $S$ of height $n-d$. Then,
$\dim(S/I)=d$. Let $\mathfrak{p}$ be a minimal prime of $I$ and let 
$
A_\mathfrak{p}:=(a_{i,j})_{x_j\notin\mathfrak{p}}
$
be the $d\times(n-{\rm ht}(\mathfrak{p}))$ submatrix of 
$A=(a_{i,j})$. Then, one has $n-{\rm ht}(\mathfrak{p})\leq n-{\rm
ht}(I)=d$, and we can adjoin columns to $A_\mathfrak{p}$ to form a
$d\times d$ submatrix $\overline{A}_\mathfrak{p}$ of $A$. Hence, by
MDS criterion (Theorem~\ref{mds-minors}),
$\det(\overline{A}_\mathfrak{p})\neq 0$. Thus, the columns of
$\overline{A}_\mathfrak{p}$ are linearly independent and so are the
columns of ${A}_\mathfrak{p}$. Therefore, $\overline{A}_\mathfrak{p}$
has full columns rank and, by the rank criterion
(Theorem~\ref{rank-criterion}), $\{f_1,\ldots,f_d\}$ is a linear system of parameters of
$S/I$.
\end{proof}

\section{Systems of parameters for edge rings of
graphs}\label{section-graphs}

\begin{definition}\label{0-1-sop-def} Let $I\subset S$ be a monomial
ideal. If $\underline{f}=\{f_1,\ldots,f_d\}$,
$f_i=\sum_{j=1}^na_{i,j}x_j$, $a_{i,j}\in K$, $i=1,\ldots,d$, is a system of parameters for
$S/I$ and the coefficient matrix $A=(a_{i,j})$ is a 0-1 matrix, we say that
$\underline{f}$ is a \textit{0-1 system of parameters} for $S/I$.
 \end{definition}

\begin{lemma}\label{sop-covers-vg} Let $\mathcal{C}$ be a clutter
without isolated vertices with vertex set
$V(\mathcal{C})=\{x_1,\ldots,x_n\}$ and let $\{f_1,\ldots,f_d\}$ be a linear system of parameters
of $S/I(\mathcal{C})$. If $I(\mathcal{C})\subset\mathfrak{m}^2$, then
any variable $x_k$ occurs in at least  
one $f_j$. In particular, the supports of 0-1 linear systems of parameters
of the edge ring $S/I(G)$ of a graph $G$ cover all vertices of $G$.
\end{lemma}

\begin{proof} Since ${\rm
rad}(I(\mathcal{C}),f_1,\ldots,f_d)=\mathfrak{m}$, there is $r\geq 1$
such that $x_k^r\in(I(\mathcal{C}),f_1,\ldots,f_d)$. Then, 
$$
x_k^r=\sum_{i=1}^mg_i(x^{\alpha_i})+\sum_{i=1}^dh_if_i,
$$
where the $g_i$'s and $h_i$'s are in $S$, and the $x^{\alpha_i}$'s are
squarefree monomials in $I(\mathcal{C})$ of degree $\geq 2$. If $x_k$ does not 
occurs in any of the $f_i$'s, making $x_\ell=0$ for $\ell\neq k$ in the
equation above, we
get that $x_k^r=0$ because $I(\mathcal{C})\subset\mathfrak{m}^2$ and
each $x^{\alpha_i}$ contains a variable different from $x_k$, 
a contradiction.   
\end{proof}


\begin{definition}\label{clique}
Let $G$ be a graph. A \textit{clique} of $G$ is a subset $A$ of $V(G)$
such that the induced subgraph $G[A]$ is a complete graph. The {\it
clique number\/} of $G$, denoted by $\omega(G)$,  
is the size of the largest clique of $G$.
A {\it coloring\/} of the vertices of $G$ is an assignment
of colors to the vertices of $G$ in such a way that adjacent vertices
have distinct colors. The {\it chromatic
number\/} of $G$, denoted by
$\chi(G)$,
is the minimal number of colors in a 
coloring of $G$. 
\end{definition}

The clique number and the chromatic number are related by 
$\omega(G)\leq\chi(G)$. 

\begin{definition}\label{perfect-def}
A graph $G$ is {\it perfect\/} if 
$\omega(H)=\chi(H)$ for every induced subgraph $H$ of $G$. 
\end{definition}
Note that the chromatic number $\chi(G)$ of a graph $G$ is the minimum number of independent
sets needed in a partition of the vertex set $V(G)$ of $G$.  
To compute chromatic numbers we use the algebraic method of \cite{FHV}
and \textit{Macaulay}$2$ \cite{mac2}.  

\begin{proposition}\label{beta_0=2}
Let $G$  be a graph with $\beta_0(G)=2$ and vertex set $V(G)=\{x_1, \ldots, x_n\}$.
Then, $S/I(G)$ has a  0-1
linear system 
of parameters if and only if $V(G)$ has a partition in three or in
two cliques of $G$. In particular, if $\overline{G}$ is the complement
of $G$ and $\chi(\overline{G})>3$, then
$S/I(G)$ does not have a 0-1 linear system of parameters. 
\end{proposition}
 \begin{proof}
$\Leftarrow$) By hypothesis, we have that $V(G)=B_1 \cup B_2$ or  $V(G)=B_1 \cup B_2 \cup B_3$, where 
$B_1, B_2, B_3$ are cliques of $G$.  We define the sets $A_1$ and
$A_2$ in the following way: 
\begin{align*}
&A_1=B_1 \mbox{ and }A_2=B_2 \mbox{ if } V(G)=B_1 \cup B_2,\\ 
&A_1=B_1 \cup B_3 \mbox{ and }A_2=B_2 \cup B_3\mbox{ if }V(G)=B_1 \cup B_2 \cup B_3. 
\end{align*}
\quad We let $f_i:=\sum_{x_j \in A_i}
x_j=\sum_{j=1}^na_{i,j}x_j$, i.e., 
$a_{i,j}=1$ if $x_j \in A_i$  and $a_{i,j}=0$ if $x_j \notin A_i$. 
We let $D$ be a maximal independent set of $G$ and consider the
$2 \times \vert D \vert$ submatrix 
$
A_D:=(a_{i,j})_{x_j\in D}
$
of the $2 \times n$ matrix $A:=(a_{i,j})$. If $\vert D \vert=1$, then
rank($A_D$)=1, 
since $V(G)=A_1 \cup A_2$.
Now, we can assume  $D=\{x_{i_1}, x_{i_2}\}$, since $\vert D \vert \leq  \beta_0(G)=2$. 
As $B_1, B_2, B_3$ are cliques, we can assume $x_{i_1} \in B_1$,
$x_{i_2} \in B_2$ or $x_{i_1} \in B_1$, 
$x_{i_2} \in B_3$. Thus, $A_D$ has one of the following forms:
\begin{align*}
\begin{matrix}
\; & x_{i_1} & x_{i_2}  \\
A_1 & 1 & 0 \\
A_2 & 0 & 1 \\
\end{matrix}
\quad \mbox{or} \quad
\begin{matrix}
\; & x_{i_1} & x_{i_2}\\
A_1 & 1 & 1  \\
A_2 & 0 & 1  \\
\end{matrix}
\end{align*}
Hence, rank($A_D$)=2. Therefore, by Corollary~\ref{sop-clutters},
$S/I(G)$ has a 
0-1 system 
of parameters.

$\Rightarrow$) There are $f_1,f_2\in S$ such that $f_i=\sum_{j=1}^na_{i,j}x_j$ with
$a_{i,j}\in \{0,1\}$ for $i=1, 2$, since  
$\dim(S/I(G))=\beta_0(G)=2$. We take $A_i=\{x_j \mbox{ } \vert
\mbox{ } a_{i,j}=1\}$. By Corollary~\ref{sop-clutters}, for each independent set $D$ of $G$
 rank($A_D$)$=\vert D \vert$, where $A_D:=(a_{i,j})_{x_j\in D}$. 
By Lemma~\ref{sop-covers-vg}, $A_1 \cup A_2 = V(G)$ and we
let $B_1=A_1 \setminus A_2$, 
 $B_2=A_2 \setminus A_1$ and $B_3=A_1 \cap A_2$. Note that $B_1 \neq
 \emptyset$ and $B_2 \neq \emptyset$ because $G$ is not a complete
 graph. Thus, it suffices to prove that  $B_1$
 and $B_2$ are cliques, and that $B_3$ is a clique or
 $B_3=\emptyset$. 
By contradiction suppose that $B_i$ is not a clique for 
some $i\in\{1,2\}$. Then, there
are $x_{j_1}$, $x_{j_2} \in B_i$ 
 such that $\{x_{j_1}, x_{j_2}\} \notin E(G)$. So,
 $D_1=\{x_{j_1}, x_{j_2}\} $ is 
a maximal independent set of $G$ and rank($A_{D_1}$)$=\vert D_1 \vert=2$.
Hence, $A_{D_1}$ has one of the following forms:
\begin{align*}
\begin{matrix}
\; & x_{j_1} & x_{j_2}  \\
A_1 & 1 & 1 \\
A_2 & 0 & 0 \\
\end{matrix}
\quad \mbox{or}\quad 
\begin{matrix}
\; & x_{j_1} & x_{j_2}\\
A_1 & 0 & 0  \\
A_2 & 1 & 1  \\
\end{matrix},
\end{align*}
a contradiction. To show that $B_3=\emptyset$ or $B_3=A_1\cap
A_2\neq\emptyset$ is a clique, we
argue by contradiction assuming there
are $x_{j_1}$, $x_{j_2} \in B_3$ 
 such that $\{x_{j_1}, x_{j_2}\} \notin E(G)$. 
Hence, $A_{D_1}$ has the form:
\begin{align*}
\begin{matrix}
\; & x_{j_1} & x_{j_2}\\
A_1 & 1 & 1  \\
A_2 & 1 & 1  \\
\end{matrix},
\end{align*} 
a contradiction. Therefore, $V(G)$ has a partition in three or in two cliques of $G$.
\end{proof}

\begin{proposition}\label{no-0-1-sop-prop} 
Let $G$ be the graph in
Example~\ref{example-comp-Grotzsch-graph}, 
Figure~\ref{comp-Grotzsch-graph}, and let $I(G)$ be its edge ideal. 
Then, $S/I(G)$ has no 0-1 linear system of parameters over any field $K$.
\end{proposition}

\begin{proof} The independence number $\beta_0(G)$ of $G$ is $2$, and 
the complement $\overline{G}$ of $G$ is the Mycielski--Gr\"otzsch graph given in 
Figure~\ref{Grotzsch-graph} whose chromatic 
number is $4$. Then, by Proposition~\ref{beta_0=2}, 
$S/I(G)$ has no 0-1 linear system of parameters. 
\end{proof}

\begin{theorem}\label{perfect-sop}
Let $G$ be a graph with vertex set $V(G)=\{x_1,\ldots,x_n\}$, 
chromatic number $\chi(G)$, complement $\overline{G}$, and
independence number 
$d=\beta_0(G)=\dim(S/I(G))$.
The following hold.
\begin{enumerate}
\item[(a)] If $A_1,\ldots,A_d$ are cliques of $G$ that form a
partition of $V(G)$, then the set of polynomials $f_i=\sum_{x_j\in
A_i}x_j$, $i=1,\ldots,d$, form a 0-1 linear system of parameters of
$S/I(G)$. 
\item[(b)] If $G$ is a perfect graph, then $S/I(G)$ has a 0-1 linear
system of parameters coming from a partition of $V(\overline{G})$ into
$\chi(\overline{G})$ 
color classes of $\overline{G}$.
\end{enumerate}
\end{theorem}

\begin{proof} (a) Let $A=(a_{i,j})$, $a_{i,j}\in\{0,1\}$, be the $d\times n$ coefficient
matrix of $\underline{f}=\{f_1,\ldots,f_d\}$, let
$F=\{x_{i_1},\ldots,x_{i_r}\}$ be a maximal independent set of $G$ with
$r$ vertices, and let $A_F:=(a_{i,j})_{x_j\in F}$ be the $d\times r$
submatrix of $A$ whose columns are indexed by $i_1,\ldots,i_r$. 
By Corollary~\ref{sop-clutters}, we need only show that $A_F$ has full
column rank equal to $r$. Each column of $A$ is equal to $e_{\ell}^\top$
for some unit vector $e_{\ell}\in K^d$ because $A_1,\ldots,A_d$ form a
partition of $V(G)$. Thus, for each $i_j$ there is $k_j$ such that 
$a_{k_j,i_j}=1$ and $a_{m,i_j}=0$ for $m\neq k_j$. Note that if $i_j\neq i_p$,
then $k_j\neq k_p$. Indeed, if $k_j=k_p$, then $a_{k_j,i_j}=1$ and
$a_{k_p,i_p}=1$ are in the $k_j$-th row of $A_F$, and then $x_{i_j}$, 
$x_{i_p}$ are in the clique $A_{k_j}$, that is, $x_{i_j}$, 
$x_{i_p}$ are adjacent, a contradiction since $F$ is an independent
set of $G$. Therefore, the columns of $A_F$ are linearly independent and
${\rm rank}(A_F)=r$.

(b) The complement of $G$ is a perfect graph \cite{chvatal,lovasz}. Then,
$\chi(\overline{G})=\omega(\overline{G})=\beta_0(G)=d$. Then, there
are independent sets $A_1,\ldots,A_d$ of $\overline{G}$, which are the
color classes of $\overline{G}$, that form a partition of
$V(\overline{G})$. Thus, $A_1,\ldots,A_d$ are cliques of $G$ that form
a partition of $V(G)$. Hence, by part (a), the set of polynomials $f_i=\sum_{x_j\in
A_i}x_j$, $i=1,\ldots,d$, form a 0-1 linear system of parameters of
$S/I(G)$.
\end{proof}

Let $G$ be a graph. The {\it vertex
covering number\/} of $G$, denoted $ \alpha_0(G)$,  is the number of
vertices in any smallest
vertex cover of $G$. A set of edges in $G$ is called
{\it independent} or a {\it matching} if no two of them have a vertex
in common.  The {\it matching number\/} of $G$, denoted $\beta_1{(G)}$, is the 
number of edges in any largest independent set of edges. 

\begin{definition}\label{konig-def}
A graph $G$ is 
called a \textit{K\"onig graph} if its matching number $\beta_1(G)$ is
equal to its covering number $\alpha_0(G)$. 
\end{definition}


\begin{lemma}\label{konig-partition}
Let $G$ be a K\"onig graph. Then, its vertex set has a
partition
\begin{equation}\label{sep13-26}
V(G)=\{x_1,\ldots,x_n\}=\displaystyle(\cup_{i=1}^{\alpha_0(G)}A_i)\cup(\cup_
{x_i\notin \cup_{i=1}^{\alpha_0(G)}A_i}\{x_i\}),
\end{equation}
into $\beta_0(G)$ cliques, where $A_1,\ldots,A_{\alpha_0(G)}$ are the edges of a 
maximum matching of $G$. 
\end{lemma}

\begin{proof} Since $\beta_1(G)=\alpha_0(G)$, there are independent
edges $A_1,\ldots,A_{\alpha_0(G)}$ of $G$. Then, one has a partition 
as in Eq.~\eqref{sep13-26}. The number of sets in this partition is 
$$\alpha_0(G)+(n-2\alpha_0(G))=n-\alpha_0(G)=\beta_0(G),$$
and the proof is complete.
\end{proof}

\begin{proposition}\cite{2021-HerzogMoradi}\label{konig-0-1-sop}
If $G$ is a K\"onig graph, then for any maximum matching
$A_1,\ldots,A_{\alpha_0(G)}$ of $G$, the ring $S/I(G)$ has a 0-1 linear system of
parameters $\{f_1,\ldots,f_d\}$, $d=\beta_0(G)$, such that 
$f_i=\sum_{x_j\in A_i}x_j$ for $i=1,\ldots,\alpha_0(G)$ and
$f_i$ is a variable not in $A_1\cup\cdots\cup A_{\alpha_0(G)}$ for $i>\alpha_0(G)$.   
\end{proposition}

\begin{proof} The result follows from Theorem~\ref{perfect-sop}(a) and 
Lemma~\ref{konig-partition}  
\end{proof}

\begin{definition}\label{star-sop}
Let $G$ be a graph and let $\{z_1, z_2, \ldots, z_d\} \subset V(G)$.
Then, a 0-1 linear system of parameters $\{f_1, \ldots, f_d\}$ of $S/I(G)$ 
is a \textit{star-system of parameters} associated to $\{z_1, z_2, \ldots, z_d\}$  if 
for each $i$, we can write $f_i=z_i + \sum_{x_j  \in A_i} x_j$, where
$A_i \subset N_G(z_i)$, 
that is, $\{z_i,x_j\}_{x_j\in A_i}$ are the edges
of a star of $G$ with center $z_i$. 
 \end{definition}

\begin{lemma}\label{star-poly}
Let $G$ be a graph, let $z \in V(G)=\{x_1, \ldots, x_n\}$, and
let $f=z + \sum_{x_j  \in A} x_j$, where $A \subset N_G(z)$. Then, 
$z \in {\rm rad}(I(G), f)$.
\end{lemma}
\begin{proof}
We set $J:={\rm rad}(I(G), f)$. Then, $zf=z^2 + \sum_{x_j  \in A} z x_j \in J$.
As $x_j  \in A$, we have that $\{z, x_j \}\in E(G)$
implies  $z x_j \in I(G)\subset J$. 
Thus, $z^2 \in J$ since $zf \in J$. Hence, $z \in J$.
\end{proof}

\begin{proposition}\label{odd-cycle}
If $G$ is an odd cycle, then $S/I(G)$ has a star-system of parameters.
\end{proposition}
\begin{proof}
We set $G=C_n=\{x_1,x_2, \ldots, x_n\}$, where $n=2d+1$,  then
$d=\beta_{0}(G)=d$. Now, we will prove  
that $f_1, \ldots f_d$ is a 0-1 system of parameters, where
$f_j=x_{2j-1}+x_{2j}+x_{2j+1}$ for $j=1, \dots, d$. 
By Lemma~\ref{star-poly}, 
$$
x_{2j} \in {\rm rad}(I(G), f_j) \subset J:= {\rm rad}(I(G), f_1, \ldots, f_d),
$$
since $x_{2j-1}, x_{2j+1} \in N_G(x_{2j})$, for $j=1, \dots, d$. Consequently, 
$g_j:=x_{2j-1}+x_{2j+1} \in J$ for $j=1, \dots, d$. This implies,
$f=\sum_{j=1}^d(-1)^{j+1}{g_j} \in J$. But $f=x_1+x_n$ if $d$ is odd and
$f=x_1-x_n$ if $d$ is even. Thus, $x_1 \in J$, since $x_1f \in J$ and
$x_1x_n \in I(G)\subset J$. Hence inductively, $x_{2j+1} \in J$, for $j=1, \dots, d$, since
$g_j \in J$ and $x_1 \in J$. Therefore, $\{x_1, \ldots, x_n\} \in J$ and
$f_1, \ldots, f_d$ is a star-system of parameters, since 
$\{x_2, \ldots, x_{2d}\}$ is an independent set and $x_{2j-1},
x_{2j+1} \in N_G({x_{2j}})$
for $j=1, \dots, d$.
\end{proof}

\begin{proposition}\label{antihole-star-sop}
Let $G=\overline{C_n}$ be the complement of a cycle $C_n$ of length $n\geq
4$. If $\{z_1, z_2\}$ is an independent set of $G$, then
there is a star-system of parameters associated to $\{z_1, z_2\}$.
\end{proposition}
\begin{proof}
Let $H=C_n$ be a cycle $H=\{x_1, \ldots, x_n\}$ with $n \geq 4$.
Thus, the independent sets of $G$ are the edges of $H$. So, we can
assume $\{z_1, z_2\}=\{x_1, x_2\}$. 
First suppose $n$ is even, then $V(G)=B_1 \cup B_2$ where  $B_1=\{x_i
\mbox{ } \vert \mbox{ } i \mbox{ is odd} \}$ 
and $B_2=\{x_i \mbox{ } \vert \mbox{ } i \mbox{ is even} \}$. Also
$B_1$ and $B_2$ are cliques of $G$. 
Hence, by the proof of a  Proposition~\ref{beta_0=2},  $\{f_1, f_2\}$
is a system of parameters 
where  $f_i:=\sum_{x_j \in B_i} x_j$. Furthermore, $x_i \in B_i
\subset N_G[x_i]$ for $i=1, 2$. 
Now if $n$ is odd, then $V(G)=B'_1 \cup B'_2 \cup B'_3$ where
$B'_1= \{x_i \mbox{ } \vert \mbox{ } i \mbox{ is odd} \}  \setminus \{x_n\}$,
$B'_2= (\{x_i \mbox{ } \vert \mbox{ } i \mbox{ is even} \} \setminus
\{x_{n-1}\}) \cup \{x_n\}$ and $B'_3=\{x_{n-1}\}$.  
Also, $B'_1$, $B'_2$ and $B'_3$ are cliques of $G$.
Hence, by the proof of  Proposition~\ref{beta_0=2},  $\{f_1, f_2\}$ is a system of parameters
where  $f_i:=\sum_{x_j \in A_i} x_j$, $A_1=B'_1 \cup B'_3$, 
$A_2=B'_2 \cup B'_3$, and we have that $x_i \in A_i \subset N_G[x_i]$
for $i=1, 2$.  Thus, in both cases $\{f_1,f_2\}$ is an star-system of
parameters associated to $\{z_1,z_2\}$.
\end{proof}

In Example~\ref{no-start-sop}, we give a graph $G$ with
$\beta_0(G)=2$ and an independent set $\{z_1, z_2\}$  
such that $S/I(G)$ has a 0-1 system of parameters, but does not
have a star system of parameter
associated to $\{z_1, z_2\}$.
Thus, Proposition~\ref{antihole-star-sop} is not true  if $G$ is
not 
the complement of a cycle. 

\begin{definition}\label{reducible-def}
A collection of subgraphs $G_1, \ldots, G_s$ of a graph $G$ such that
$\{V(G_i)\}_{i=1}^s$ is 
a partition of $V(G)$ with $s \geq 2$ and $\alpha_0(G)=\sum_{i=1}^s
\alpha_0(G_i)$, 
is called an $\alpha_0$-\textit{reduction} of $G$. Furthermore, a graph is called 
$\alpha_0$-$\textit{reducible}$ if $G$ has an $\alpha_0$-reduction
and $\alpha_0$-\textit{irreducible} if it does not have an $\alpha_0$-reduction.  
\end{definition}

\begin{remark}\label{b_0-reduction}
Let $G_1, \ldots, G_s$ be an $\alpha_0$-reduction of a graph $G$ and
let $n_i=\vert V(G_i)\vert$. 
Then,  $\alpha_0(G_i) = n_i\ - \beta_0(G_i)$
and since $\alpha_0(G)=\sum_{i=1}^s\alpha_0(G_i)$, we have
$$
\vert V(G) \vert -\beta_0(G)=\alpha_0(G)=\sum_{i=1}^s\alpha_0(G_i)=\sum_{i=1}^s(n_i-\beta_0(G_i))=\sum_{i=1}^sn_i - \sum_{i=1}^s\beta_0(G_i).
$$
Hence, $\beta_0(G)=\sum_{i=1}^s\beta_0(G_i)$ since
$\{V(G_i)\}_{i=1}^s$ is a partition of $V(G)$ and $|V(G)|=\sum_{i=1}^sn_i$.
\end{remark}

\begin{proposition}\label{reduction-star-0-1}
If $G$ has a $\alpha_0$-reduction $G_1, \ldots, G_s$ such that the
edge ring of each $G_i$  has a 0-1 system (resp. star-system) 
of parameters, then the edge ring of $G$ has a 0-1 system (resp. star-system) of parameters.
\end{proposition}
\begin{proof}
Let $\{f_1^i, \ldots, f_{d_i}^i \}$ be a 0-1 system of parameters of $G_i$,
where $d_i=\beta_0(G_i)$. We let $A_i$ be the $d_i \times n_i$ matrix whose entries 
are the coefficients of $f_1^i, \ldots, f_{d_i}^i $. Now, we will prove that 
$\underline{f}=\cup_{i=1}^s \{ f_1^i, \ldots, f_{d_i}^i  \}$ is a system of parameters of $G$. 
First, by Remark~\ref{b_0-reduction},  
$$
\vert \underline{f} \vert= \sum_{i=1}^s d_i= \sum_{i=1}^s \beta_0(G_i)=\beta_0(G),
$$
since $\{V(G_1), \ldots, V(G_s)\}$ is a partition of $V(G)$.
Consequently, the $\beta_0(G) \times n$ matrix 
whose entries are the coefficients of the elements of $\underline{f}$ has the
form
\[
\begin{pmatrix}
A_1&0&0&\cdots&0\\
0&A_2&0&\cdots&0\\
0&0&A_3&\cdots&0\\
\vdots&\vdots&\vdots&\ddots&\vdots\\
0&0&0&\cdots&A_s\\
\end{pmatrix}.
\]
If $F$ is an independent set of $G$, then $\{F_1, \ldots, F_s\}$ is a partition of $F$,
where $F_i=F \cap V(G_i)$. This implies,
\[
A_F=
\begin{pmatrix}
A_{F_1}&0&0&\cdots&0\\
0&A_{F_2}&0&\cdots&0\\
0&0&A_{F_3}&\cdots&0\\
\vdots&\vdots&\vdots&\ddots&\vdots\\
0&0&0&\cdots&A_{F_s}\\
\end{pmatrix}.
\]
Hence, rank$(A_F)=\sum_{i=1}^s$rank$(A_{F_i})$ and $A_F$ has full column rank,
since $A_{F_i}$ has full column rank. Therefore, by
Corollary~\ref{sop-clutters},
$\underline{f}$ is a system of parameters of $S/I(G)$.
\end{proof}

\begin{conjecture}\label{conjecture-star}
If $S/I(G)$ has a 0-1 system of parameters, then $S/I(G)$ has a star-system of parameters
associated to some maximal independent set $\{z_1, \ldots, z_d\}$ of $G$. 
\end{conjecture}
\section{Examples}\label{examples-section}

\begin{example}
Let $B$ be a matrix of full column rank $r$ with entries in
$\mathbb{R}$. Then, by Lemma~\ref{pseudo-inverse}, the matrix 
given by $B^\dag:=(B^\top B)^{-1}B^\top$ is well defined and
$B^{\dag}B=I_r$, that is, $B$ is left-invertible. The matrix
$B^{\dag}$  
called a \text{\it pseudo inverse} of $B$. If $c,b$ are vectors
and $Bc=b$, then $c=B^{\dag}b$. 
\end{example}

\begin{example}
Let $S=K[x_1,x_2,x_3,x_4,x_5]$ and let $I=(x_1x_2,x_2x_3,x_3x_4,x_4x_5,x_5x_1).$
This is the edge ideal of the cycle graph $C_5$. Its minimal primes are
\[
\begin{aligned}
\mathfrak p_1&=(x_1,x_3,x_5),&
\mathfrak p_2&=(x_1,x_3,x_4),&
\mathfrak p_3&=(x_1,x_2,x_4),&
\mathfrak p_4&=(x_2,x_4,x_5),&
\mathfrak p_5&=(x_2,x_3,x_5).&
\end{aligned}
\]
Since all these primes have height $3$, we have
\[
\dim(S/I)=5-3=2.
\]	
Consider the linear forms
\[
f_1=x_1+x_2+x_3+x_4+x_5\ \mbox{ and }\ f_2=x_1+2x_2+3x_3+4x_4+5x_5.
\]
Their coefficient matrix is 
\[
A=
\begin{pmatrix}
1&1&1&1&1\\
1&2&3&4&5
\end{pmatrix}.
\]
For the minimal primes listed above, the corresponding submatrices 
in Theorem~\ref{rank-criterion} are
\[
A_{\mathfrak p_1}=
\begin{pmatrix}
1&1\\
2&4
\end{pmatrix},
\quad
A_{\mathfrak p_2}=
\begin{pmatrix}
1&1\\
2&5
\end{pmatrix},
\quad
A_{\mathfrak p_3}=
\begin{pmatrix}
1&1\\
3&5
\end{pmatrix},
\quad
A_{\mathfrak p_4}=
\begin{pmatrix}
1&1\\
1&3
\end{pmatrix},
\quad
A_{\mathfrak p_5}=
\begin{pmatrix}
1&1\\
1&4
\end{pmatrix}.
\]
Their determinants are $2, 3, 2, 2, 3$, respectively.
Therefore, if $\operatorname{char}(K)\notin\{2,3\}$, all the matrices
$A_{\mathfrak p_i}$ have full column rank. Hence, by Theorem~\ref{rank-criterion},
\[
\{f_1,f_2\}
\]
is a system of parameters for $S/I$.
On the other hand, this particular pair is not a system of parameters
over fields of characteristic $2$ or $3$. Thus, the example illustrates
that the existence of a given linear system of parameters may depend on
the characteristic of the field.
\end{example}

\begin{example}\label{mono-dim=2} Let $S=K[x_1,x_2,x_3,x_4]$ be a polynomial ring over a
field $K$. Consider the monomial ideal  
\begin{align*}
I=&(x_2^5,x_2^3x_3^2,x_1^4x_3^3,x_1^3x_2x_3^3,x_2^2x_3^3,x_2x_3^4,x_2^3x_4,x_2x_3^3x_4^2)\\
=&(x_1^4,x_2)\cap(x_2^3,x_3^3)\cap(x_1^3,x_2^2,x_3^4,x_4^2)\cap
(x_2^5,x_3^2,x_4).
\end{align*}
\quad The associated primes of $I$ are $\{x_1,x_2\}$, $\{x_2,x_3\}$,
$\{x_2,x_3,x_4\}$, and $\mathfrak{m}=(x_1,x_2,x_3,x_4)$, and the
minimal primes of $I$ are $\mathfrak{p}_1=\{x_1,x_2\}$ and
$\mathfrak{p}_2=\{x_2,x_3\}$. If $K$ is a field of ${\rm char}(K)=0$,
by Corollary~\ref{coro-dim=2}, the polynomials 
$$
f_1=x_1+2x_2+3x_3+4x_4\ \mbox{ and }\ f_2= x_2+2x_3+3x_4
$$
form a system of parameters for $S/I$, and by Theorem~\ref{rank-criterion}, 
it follows that $\{x_4,x_1+x_3\}$ is a 0-1 system of parameters 
over any field $K$. 
\end{example}

\begin{example}\label{sop-from-ff}
Let $K$ be the finite field $\mathbb{F}_5$, let $S=K[x_1,x_2,x_3,x_4,x_5]$,
let  
$$
A=\begin{pmatrix}1& 0& 0& 1& 1\\
 0& 1& 0& 1& 2\\
0& 0& 1& 1& 3
\end{pmatrix},
$$
and let $C$ be the linear code generated by the rows of $A$. Using 
the rank criterion (Theorem~\ref{rank-criterion}) and the MDS
criterion (Theorem~\ref{mds-minors}), one has that $C$ is an
$[5,3,3]_5$ linear code, and 
$$f_1=x_1+x_4+x_5,\ f_2=x_2+x_4+2x_5,\  
f_3=x_3+x_4+3x_5, 
$$
are a linear system of parameters for the quotients rings of any of the following monomial
ideals of $S$ of height $2$:
\begin{align*}
&I=(x_1^2x_2,x_2x_3,x_1x_3^2,x_1x_2x_4,x_1x_4^2,x_2x_4x_5),\\
&{\rm rad}(I)=(x_1x_2,x_1x_3,x_2x_3,x_1x_4,x_2x_4x_5),\quad
J=(x_1,x_2),\\
&L=(x_1x_2x_3x_4,x_1x_2x_3x_5,x_1x_2x_4x_5,x_1x_3x_4x_5,x_2x_3x_4x_5).
\end{align*}
\quad If $K=\mathbb{F}_2$, then $C$ is a $[5,3,2]_2$ linear code and
$\{x_3,x_2+x_5,x_1+x_2+x_4\}$ is a system of parameters for $S/I$.
This example corresponds to Procedure~\ref{procedure-MAGMA}.
\end{example}

\begin{example}
Let $S=K[x_1,x_2,x_3,x_4]$ be a polynomial ring over a 
field $K$ and let 
\begin{align*}
I=(x_1x_2x_3,x_1x_2x_4,x_1x_3x_4,x_2x_3x_4)=\bigcap_{1\leq i<j\leq
4}(x_i,x_j).
\end{align*}
\quad Then, $\dim(S/I)={\rm ht}(I)=2$. If ${\rm char}(K)=0$, by
Corollary~\ref{coro-dim=2},  
$$f_1=x_1+2x_2+3x_3+4x_4,\ \ f_2=x_2+2x_3+3x_4
$$
is a system of parameters for $S/I$. If $K=\mathbb{Q}$, there is no
0-1 system of parameters for $S/I$, this follows using the zeroOneSOP
function of Procedure~\ref{0-1-sop}. If $K=\mathbb{Z}_2$, then there
are no linear systems of parameters for $S/I$, this follows applying
the rank criterion (Theorem~\ref{rank-criterion}) to all
possible systems of parameters of $S/I$ or equivalently finding the
$2\times 2$ minors of all $2\times 4$ matrices of rank $2$ over $\mathbb{F}_2$. 
\end{example}

\begin{example}
Let $G$ be the graph with vertex set $V(G)=\{x_1, x_2, x_3, x_4, x_5, z_1, z_2\}$ which
consists of  a $5$-cycle $C_5=\{x_1,x_2,x_3, x_4,x_5\}$ 
and $N_G(z_1)=N_G(z_2)=\{x_1,x_2,x_3, x_4,x_5\}$, and let $S=K[V(G)]$.
We take a maximal independent set $F$. If $z_1 \in F$ or $z_2 \in
F$, then $F=\{z_1, z_2\}$. Now, if
$z_1, z_2 \notin F$, then $F \subset \{x_1,x_2,x_3, x_4,x_5\}=V(C_5)$, implies
$\vert F \vert=2$, since $F$ is maximal. Hence, $\beta_0(G)=2$ and $G$ is unmixed.
Furthermore, $V(G)$ has a partition in cliques 
\begin{align*}
&B_1 \cup B_2 \cup B_3,\ \mbox{ where }\ 
B_1=\{z_1, x_1, x_2\},\  B_2=\{z_2, x_4, x_5\},\ B_3=\{x_3\},
\end{align*}
then $S/I(G)$ has a  
star-system of parameters by Proposition~\ref{beta_0=2}. 
On the other hand $\{z_1, z_2\}$ is a maximal independent set of $G$. We take
$f_1, f_2 \in S$ such that
$f_1=z_1+\sum_{x_j  \in A_1} x_j$ and $f_2=z_2+\sum_{x_j  \in A_2} x_j$,
where $A_1=\{x_1, x_2, x_3\} \subset N_G(z_1)$ and $A_2=\{x_1, x_4,
x_5\} \subset N_G(z_2)$.
By the last argument if $F$ is a maximal independent set, then 
$$F=\{a,b\} \in \{ \{z_1,z_2\},  \{x_1,x_3\},  \{x_1,x_4\},
\{x_2,x_4\},  \{x_2,x_5\},  \{x_3,x_5\}  \}.$$
In each case ${\rm rank}(A_{F})=2$. Hence, by
Corollary~\ref{sop-clutters}, 
$\{f_1,f_2\}$ is a star system of parameters of $S/I(G)$. Note that
$G$ is $\alpha_0$-irreducible since $V(G)$ has no partition in two
cliques and $\beta_0(G)=2$. 
\end{example}

\begin{example}\label{explicit-lsop-ex}
Let $S=K[x_1,\ldots,x_8]$ be a polynomial ring over a field $K$ and
let $I$ be the ideal 
$$I=(x_1x_2,x_2x_3,x_3x_4,x_4x_5,x_5x_6,x_6x_7,x_7x_8,x_1x_8,x_2x_6,x_3x_7).$$
\quad One has $\dim(S/I^k)=3$ and ${\rm ht}(I^k)=5$ for all $k\geq 1$. 
Let $K=\mathbb{Q}$ be the field of rational numbers, letting 
\begin{align*}
&\lambda_1=0,\lambda_2=1,\lambda_3=2,\lambda_4=3,\lambda_5=4,\lambda_6=5,\lambda_7=6,\lambda_8=7,\\
&f_1=x_1+x_2+x_3+x_4+x_5+x_6+x_7+x_8,\\
&f_2=
\lambda_1x_1+\lambda_2x_2+\lambda_3x_3+\lambda_4x_4+\lambda_5x_5+\lambda_6x_6+\lambda_7x_7+\lambda_8x_8,\\
&f_3=
\lambda_1^2x_1+\lambda_2^2x_2+\lambda_3^2x_3+\lambda_4^2x_4+\lambda_5^2x_5+\lambda_6^2x_6+\lambda_7^2x_7+\lambda_8^2x_8.
\end{align*}
\quad By Theorem~\ref{explicit-lsop}, $\underline{f}=\{f_1,f_2,f_3\}$ is a system of
parameters for $S/I^k$ for all $k\geq 1$ because the minimal primes of $I$ and $I^k$ are the same
for all $k\geq 1$. Using Procedure~\ref{procedure-NN} for
\textit{Macaulay}$2$ it follows that for 
$k=1$ and $k=2$, one has 
$$
(I^k:f_1)=I^k,\ \ ((I^k,f_1):f_2)=(I^k,f_1),\ \
((I^k,f_1,f_2):f_3)=(I^k,f_1,f_2),
$$
that is, $\underline{f}$ is a regular sequence of $S/I^k$, and
consequently $S/I^k$ is
Cohen--Macaulay for $k=1$ and $k=2$. If $k=3$, $\underline{f}$ is not
a regular sequence and $S/I^k$ is not Cohen--Macaulay. If $k=4$, the
maximal ideal $\mathfrak{m}$ of $S$ is an associated prime of $S/I^4$,
and by persistence property of edge ideals of graphs \cite{ass-powers}, $\mathfrak{m}$
is an associated prime of 
$S/I^k$ for all $k\geq 4$. Hence, $S/I^k$ has depth equal to $0$ for
all $k\geq 4$ and $S/I^k$ is not Cohen--Macaulay for all $k\geq 4$.
Similar statements hold if we let $K=\mathbb{F}_8$ be a finite field, 
with multiplicative group $K^*=(a)$ generated by $a$, and 
\begin{align*}
&\lambda_1=0_S,\lambda_2=a,\lambda_3=a^2,\lambda_4=
a^3,\lambda_5=a^4,\lambda_6=a^5,\lambda_7=a^6,\lambda_8=a^7,\\
&f_1=x_1+x_2+x_3+x_4+x_5+x_6+x_7+x_8,\\
&f_2=
\lambda_1x_1+\lambda_2x_2+\lambda_3x_3+\lambda_4x_4+\lambda_5x_5+\lambda_6x_6+\lambda_7x_7+\lambda_8x_8,\\
&f_3=
\lambda_1^2x_1+\lambda_2^2x_2+\lambda_3^2x_3+\lambda_4^2x_4+\lambda_5^2x_5+\lambda_6^2x_6+\lambda_7^2x_7+\lambda_8^2x_8.
\end{align*}
\end{example}

\begin{example}\label{no-0-1-sop} Let $S=\mathbb{Q}[x_1,x_2,x_3,x_4,x_5]$
and let $I$ be the monomial ideal 
$$I=(x_1x_2x_3x_4,x_2x_3x_4x_5,x_1x_3x_4x_5,x_1x_2x_4x_5,x_1x_2x_3x_5).$$
\quad Applying the function zeroOneSOP of 
Procedure~\ref{0-1-sop} to $I$, we obtain 
that there is no 0-1 linear system of parameters for $S/I$.
\end{example}
\begin{example}\label{Grotzsch-example} Let $H$ be the
Mycielski--Gr\"otzsch graph given in 
Figure~\ref{Grotzsch-graph}, 
that arises in coloring problems \cite[Chapter~12, Fig.
12.8]{soifer}. The graph $H$ has $11$ vertices, $20$ edges,
independence number $\beta_0(H)=5$, covering number $\alpha_0(H)=6$,
chromatic number equal to $4$, and is not a perfect graph.
The
edge ideal $I(H)$ of $H$ is given by  
\begin{align*}
I(H)=&(x_1x_2,x_2x_3,x_3x_4,x_1x_5,x_4x_5,x_2x_6,x_5x_6,x_1x_7,x_3x_7,x_2x_8,x_4x_8,x_3x_9,\\
&x_5x_9,x_1x_{10},x_4x_{10},x_6x_{11},x_7x_{11},x_8x_{11},x_9x_{11},x_{10}x_{11}).
\end{align*}
Using  the zeroOneSOP function of Procedure~\ref{0-1-sop}, we obtain the following 0-1 linear
system of parameters for $S/I(H)$:
\[
f_1=x_6, f_2=x_3+x_7, f_3=x_1+x_2+x_8, f_4=x_4+x_5+x_9,
f_5=x_1+x_4+x_{10}+x_{11}.
\]
\begin{figure}[H]
\centering
\begin{tikzpicture}[scale=0.8,
	vertex/.style={
		circle, 
		draw, 
	fill=white, 
	minimum size=6mm, 
	inner sep=0pt}, 
	edge/.style={thick}
	]
	\node[vertex] (x1) at (90:3cm)  {$x_1$};
	\node[vertex] (x2) at (18:3cm)  {$x_2$};
	\node[vertex] (x3) at (-54:3cm) {$x_3$};
	\node[vertex] (x4) at (-126:3cm){$x_4$};
	\node[vertex] (x5) at (162:3cm) {$x_5$};
	
	\node[vertex] (x6) at (90:1.35cm)   {$x_6$};
	\node[vertex] (x7) at (18:1.35cm)   {$x_7$};
	\node[vertex] (x8) at (-54:1.35cm)  {$x_8$};
	\node[vertex] (x9) at (-126:1.35cm) {$x_9$};
	\node[vertex] (x10) at (162:1.35cm)  {$x_{10}$};
	
	\node[vertex] (x11) at (0,0) {$x_{11}$};
	
	\draw[edge] (x1) -- (x2);
	\draw[edge] (x2) -- (x3);
	\draw[edge] (x3) -- (x4);
	\draw[edge] (x4) -- (x5);
	\draw[edge] (x5) -- (x1);
	\draw[edge] (x1) -- (x7);
	\draw[edge] (x1) -- (x10);
	
	\draw[edge] (x2) -- (x6);
	\draw[edge] (x2) -- (x8);
	
	\draw[edge] (x3) -- (x7);
	\draw[edge] (x3) -- (x9);
	
	\draw[edge] (x4) -- (x8);
	\draw[edge] (x4) -- (x10);
	
	\draw[edge] (x5) -- (x9);
	\draw[edge] (x5) -- (x6);
	
	\draw[edge] (x11) -- (x6);
	\draw[edge] (x11) -- (x7);
	\draw[edge] (x11) -- (x8);
	\draw[edge] (x11) -- (x9);
	\draw[edge] (x11) -- (x10);
\end{tikzpicture}
\vspace{2mm}
\caption{Mycielski--Gr\"otzsch graph $H$.}\label{Grotzsch-graph} 
\end{figure}
\vspace{-4mm}
\end{example}

\begin{example}\label{example-comp-Grotzsch-graph} Let $G$ be the graph of
Figure~\ref{comp-Grotzsch-graph} whose complement
$\overline{G}$ is the 
Mycielski--Gr\"otzsch graph $H$ given in Figure~\ref{Grotzsch-graph}.
The graph $G$ has $11$ vertices, $35$ edges,
independence number $\beta_0(G)=2$, covering number $\alpha_0(H)=9$,
chromatic number equal to $6$, and is not a perfect graph.
The
edge ideal $I=I(G)$ of $G$ is given by  
\begin{align*}
I(G) = &(x_1x_3, x_1x_4, x_2x_4, x_2x_5, x_3x_5,
    x_1x_6, x_1x_8, x_1x_9,
    x_2x_7, x_2x_9, x_2x_{10},
    x_3x_6, x_3x_8,\\
& x_3x_{10},
    x_4x_6, x_4x_7, x_4x_9,
    x_5x_7, x_5x_8, x_5x_{10},
    x_1x_{11}, x_2x_{11}, x_3x_{11}, x_4x_{11}, x_5x_{11},\\
&    x_6x_7, x_6x_8, x_6x_9, x_6x_{10},
    x_7x_8, x_7x_9, x_7x_{10},
    x_8x_9, x_8x_{10},
    x_9x_{10}).
\end{align*}
Using Procedure~\ref{procedure-NN} and \textit{Macaulay}$2$, we obtain that over the field
$K=\mathbb{Q}$, the following is a linear system of parameters for
$S/I(G)$: 
\begin{align*}
f_1=&\textstyle\frac{5}{3}x_1+\frac{8}{3}x_2+\frac{2}{9}x_3+\frac{10}{3}x_4+\frac{7}{10}x_5+\frac{1}{3}x_6+\frac{10}{9}x_7+\frac{6}{5}x_8+\frac{9}{8}x_9+x_{11},\\
f_2=&\textstyle\frac{4}{5}x_1+\frac{1}{8}x_2+\frac{4}{7}x_3+x_4+\frac{4}{7}x_5+\frac{6}{5}x_6+\frac{3}{10}x_7+\frac{4}{3}x_8+\frac{3}{5}x_9+x_{10},
\end{align*}
$S/I(G)$ is Cohen--Macaulay and $\{f_1,f_2\}$ is a regular system of
parameters for $S/I(G)$
\begin{figure}[H]
\centering
		\begin{tikzpicture}[scale=1.5,thick]
		\tikzstyle{every node}=[minimum width=0pt, inner sep=2pt, circle]
			\draw (-0.7,0.18) node[draw]      (0) {\tiny$x_7$};
			\draw (-0.67,-0.62) node[draw] (1) { \tiny $x_{10}$};
			\draw (0.93,-0.58) node[draw]   (2) { \tiny $x_8$};
			\draw (0.89,0.2) node[draw]        (3) { \tiny $x_6$};
			\draw (0.06,0.93) node[draw]     (4) { \tiny $x_9$};
			\draw (0.05,2.22) node[draw]      (5) { \tiny $x_{11}$};
			\draw (0.04,1.46) node[draw]      (6) { \tiny $x_4$};
			\draw (2,0.91) node[draw]            (7) { \tiny $x_1$};
			\draw (2.05,-1.17) node[draw]     (8) { \tiny $x_3$};
			\draw (-1.79,-1.26) node[draw]   (9) { \tiny $x_5$};
			\draw (-1.87,0.9) node[draw]      (10) { \tiny $x_2$};
			
			\draw[blue]  (0) edge (1);
			\draw[blue]  (0) edge (2);
			\draw[blue]  (0) edge (3);
			\draw[blue]  (0) edge (4);
			\draw[blue]  (1) edge (2);
			\draw[blue]  (1) edge (3);
			\draw[blue]  (1) edge (4);
			\draw[blue]  (2) edge (3);
			\draw[blue]  (2) edge (4);
			\draw[blue]  (3) edge (4);
			\draw[red]  (6) edge (7);
			\draw[green] (4) edge (6);
			\draw[green]  (3) edge (6);
			\draw[green]  (0) edge (6);
			\draw[red]  (6) edge (10);
			\draw[red]  (5) edge (6);
			\draw[red]  (7) edge (8);
			\draw[green]  (3) edge (8);
			\draw[green]  (1) edge (8);
			\draw[red]  (8) edge (9);
			\draw[green]  (2) edge (8);
			\draw[red]  (5) edge (7);
			\draw[red]  (5) edge (10);
			\draw[red]  (5) edge (9);
			\draw[red]  (5) edge (8);
			\draw[green]  (0) edge (10);
			\draw[green]  (4) edge (10);
			\draw[green]  (1) edge (10);
			\draw[green] (4) edge (7);
			\draw[green]  (3) edge (7);
			\draw[green]  (2) edge (7);
			\draw[green]  (2) edge (9);
			\draw[green]  (1) edge (9);
			\draw[red]  (9) edge (10);
			\draw[green]  (0) edge (9);
		\end{tikzpicture}

\vspace{2mm}
\caption{Graph $G$ is the complement of the Mycielski--Gr\"otzsch
graph. 
}\label{comp-Grotzsch-graph}  
\end{figure}
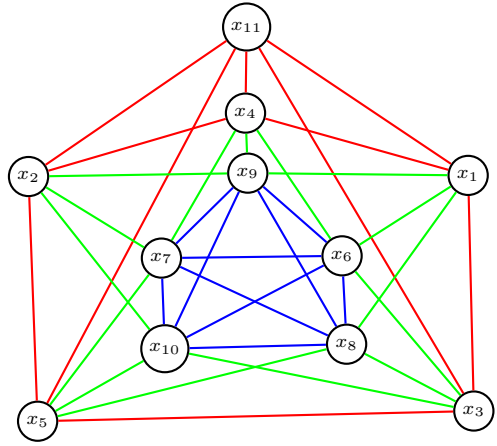

\end{example}

\begin{example} Let $C_7=\{x_1,\ldots,x_7\}$ be a $7$-cycle and let $\overline{C_7}$ be
its complement. A star system of parameters of $S/I(C_7)$ associated
to the independent set 
$\{x_2,
x_4, x_6\}$ of $C_7$ 
 is 
$$f_1=x_1+x_2+x_3,\ f_2=x_3+x_4+x_5,\ f_3=x_5+x_6+x_7,$$
where $x_2,x_4,x_6$ are the centers of the stars corresponding
to $f_1,f_2,f_3$, respectively, and a star system of parameters associated to the
independent set $\{x_1,x_2\}$ of $\overline{C_7}$ is 
$$f_1=x_1+x_3+x_5+x_6,\ f_2=x_2+x_4+x_6+x_7,$$
where $x_1,x_2$ are the centers of the stars corresponding
to $f_1,f_2$, respectively
\end{example}

\begin{example}\label{no-start-sop}
Let $G$ be the graph whose vertex set is $V(G)=\{x_1, x_2, x_3, x_4, x_5,x_6,  z_1, z_2\}$ which
consist of: a $6$-cycle $C_6=\{x_1, x_2, x_3, x_4, x_5, x_6\}$ with $3$ chords $\{x_2, x_6\}$, 
$\{x_3, x_5\}$, $\{x_3, x_6\}$ furthermore $N_G(z_1)=\{x_1, x_2, x_3, x_4\}$ and 
$N_G(z_2)=\{x_1, x_4, x_5, x_6\}$. Then, $A_1=\{z_1, x_2, x_3 \}$, $A_2=\{z_2, x_1, x_6\}$ and 
$A_3=\{x_4, x_5\}$ are cliques  and $A_1 \cup A_2 \cup A_3 = V(G)$.
We prove now that
$\beta_0(G)=2$. We take an independent set $F$. If $z_1 \in F$ or $z_2 \in F$, then 
$$
F \setminus\{z_1\} \subset V(G)\setminus N_G[z_1]= \{z_2, x_5, x_6 \} \mbox{ or } 
F \setminus \{z_2\} \subset V(G) \setminus N_G[z_2]= \{z_1, x_2, x_3 \},
$$
respectively. Thus, $\vert F \vert \leq 2$ since $\{z_2, x_5, x_6 \}$ and $\{z_1, x_2, x_3 \}$ 
are cliques. Now, assume $z_1, z_2 \notin S$, then
$F \in \{x_1, x_2, x_6\} \cup \{x_3, x_4, x_5\}$. But $\{x_1, x_2, x_6\}$ and $ \{x_3, x_4, x_5\}$  
are cliques, then $\vert F \vert \leq 2$.  Hence, $\beta_0(G)=2$ since $\{z_1, z_2\}$ is an 
independent set. Consequently, by Proposition~\ref{beta_0=2},
$S/I(G)$ has a 0-1-system of 
parameters, since $V(G)$ has a partition in three cliques. Finally, we will prove that $S/I(G)$ 
does not have a star-system of parameters associated to $\{z_1, z_2\}$.
By contradiction, suppose there is a 0-1-system of parameters $\{f_1, f_2\}$  of $S/I(G)$ 
such that $f_1=z_1 + \sum_{x_i \in A'_1}x_i$ and $f_2=z_2 + \sum_{x_i \in A'_2}x_i$ where $
A'_j \subset N_G(z_j)$ for $j=1,2$. By Lemma~\ref{sop-covers-vg}, 
$A'_1 \cup A'_2=\{x_1, \ldots, x_6\}$. Furthermore, $x_2, x_3 \in N_G(z_1) \setminus N_G(z_2)$ and 
$x_5, x_6 \in N_G(z_2) \setminus N_G(z_1)$, then $x_2, x_3 \in A'_1 \setminus A'_2$  and 
$x_5, x_6 \in A'_2 \setminus A'_1$ since $A'_j \subset N_G(z_j)$ for $j=1,2$.  
Also, we have one of the following five cases: $x_1 \in A'_1 \setminus A'_2$, 
$x_1 \in A'_2 \setminus A'_1$, $x_4 \in A'_1 \setminus A'_2$, $x_4 \in A'_2 \setminus A'_1$ or 
$x_1, x_4 \in A'_1 \cap A'_2$. But in each case, we obtain one of the following submatrix:
\[
\begin{matrix}
\; & x_1 & x_3\\
A'_1 & 1 & 1  \\
A'_2 & 0 & 0  \\
\end{matrix}
\quad
\quad
\quad
\begin{matrix}
\; & x_1 & x_5\\
A'_1 & 0 & 0  \\
A'_2 & 1 & 1  \\
\end{matrix}
\quad
\quad
\quad
\begin{matrix}
\; & x_2 & x_4\\
A'_1 & 1 & 1  \\
A'_2 & 0 & 0  \\
\end{matrix}
\quad
\quad
\quad
\begin{matrix}
\; & x_4 & x_6\\
A'_1 & 0 & 0  \\
A'_2 & 1 & 1  \\
\end{matrix}
\quad
\quad
\quad
\begin{matrix}
\; & x_1 & x_4\\
A'_1 & 1 & 1  \\
A'_2 & 1 & 1  \\
\end{matrix}
\]
This contradicts Corollary~\ref{sop-clutters}, since 
$\{ x_1, x_3 \}, \{ x_1, x_5 \}, \{ x_2, x_4 \}, \{ x_4, x_6 \}, \{
x_1, x_4 \}$ are maximal independent sets.  
Therefore, the edge ring $S/I(G)$ does not have a star-system of parameters associated to $\{z_1, z_2\}$.
\end{example}

\begin{appendix}

\section{Procedures for systems of parameters}\label{procedures-sop}

In this section, we give procedures for MAGMA \cite{magma} and \textit{Macaulay}$2$
\cite{mac2}  to determine linear systems of
parameters and algebraic and combinatorial properties of monomial
ideals and graphs. We define a function zeroOneSOP 
that finds a 0-1 linear system of parameters, if one exists.


\begin{procedure}\label{procedure-MAGMA} 
The following procedure for MAGMA \cite{magma} computes the basic parameters 
of an $[n,d,\delta]_q$ linear code $C$ from a generator matrix. This
procedure corresponds to Example~\ref{sop-from-ff} 
\begin{verbatim}
F := GF(5);
A := Matrix(F, 3, 5, [
    1,0,0,1,1,
    0,1,0,1,2,
    0,0,1,1,3]);
C := LinearCode(A);
C;
MinimumDistance(C);
Dimension(C);
Length(C);
\end{verbatim}
\end{procedure}

\begin{procedure}\label{0-1-sop}
This procedure defines a function zeroOneSOP that tests for the 
existence of a 0-1 linear system of parameters for a monomial ideal
quotient 
and returns one when it exists. This procedure is used in
Example~\ref{Grotzsch-example}.
\begin{verbatim}
clearAll;
zeroOneSOP = (I) -> (
S := ring I;V := flatten entries vars S;n := #V;d:= dim(S/I);
minimalPrimesI := minimalPrimes I;
print "==================================================";
print "Search for a 0-1 system of parameters";
print "==================================================";
variablesIndexes := toList(0..n-1);
indicesPrimos := apply(minimalPrimesI,P -> 
(varsP := flatten entries gens P;apply( varsP,
x -> position(V,y -> y === x))));
indicesRestantes := apply(indicesPrimos,P -> select(
variablesIndexes,j -> not isMember(j,P)));
listaSoportes :=flatten apply(1..n,
k -> subsets(variablesIndexes,k));
constraints := apply(indicesRestantes,C -> (apply(listaSoportes,
A -> (apply(C,j -> (if isMember(j,A) then 1_QQ else 0_QQ))))));
constructForm := A -> (sum(apply(A, j -> V#j)));
verifyFinalRank := (coleccionIndices) -> (
for t from 0 to #indicesRestantes-1 do (
C := indicesRestantes#t;m := #C; if m > 0 
then (filas := apply(coleccionIndices,i -> constraints#t#i);
Ap := matrix filas;r := rank Ap;if r =!= m then (
return false;);););return true;);
canComplete := (coleccionIndices) -> (
k := #coleccionIndices;filasRestantes := d-k; 
for t from 0 to #indicesRestantes-1 do (C := indicesRestantes#t;
m := #C;if m > 0 then (if k + filasRestantes < m then 
(return false;);
if k > 0 then (filas :=apply(coleccionIndices,
i -> constraints#t#i);Ap := matrix filas;r := rank Ap;
if r + filasRestantes < m then (return false;););););
return true;);numeroNodos := 0;numeroPodados := 0;
numeroColeccionesCompletas := 0;numeroPasaronRango := 0;
buscarColeccion =(indiceInicio,coleccionActual) -> (
numeroNodos = numeroNodos + 1;
if canComplete(coleccionActual) === false then (
numeroPodados = numeroPodados + 1;return {};);
if #coleccionActual === d then (numeroColeccionesCompletas =
numeroColeccionesCompletas + 1;
if verifyFinalRank(coleccionActual) === false then (
return {};);numeroPasaronRango =numeroPasaronRango + 1;   
return coleccionActual;);
numeroDisponibles :=#listaSoportes - indiceInicio;
numeroNecesarios :=d - #coleccionActual;
if numeroDisponibles < numeroNecesarios then (return {};);
limite :=#listaSoportes - numeroNecesarios;
for i from indiceInicio to limite do (nuevaColeccion :=
append(coleccionActual,i);
resultado :=buscarColeccion(i+1,nuevaColeccion);
if #resultado > 0 then (return resultado;););return {};);
solucion :=buscarColeccion(0,{});
if #solucion > 0 then (formasFinales :=apply(solucion,
                i -> constructForm(listaSoportes#i));
print "   A 0-1 system of parameters was found:";
        return formasFinales;) else (
print    "There is no 0-1 system of parameters";
print "==============================";return {};););
--Examples using the zeroOneSOP function
S = QQ[x1,x2,x3,x4,x5,x6,x7,x8,x9,x10,x11];
--Mycielski--Grotzsch graph
I=monomialIdeal(x1*x2,x2*x3,x3*x4,x1*x5,x4*x5,x2*x6,x5*x6,
x1*x7,x3*x7,x2*x8,x4*x8,x3*x9,x5*x9,x1*x10,x4*x10,x6*x11,
x7*x11,x8*x11,x9*x11,x10*x11)
zeroOneSOP(I)
S=QQ[x1,x2,x3,x4,x5]
I=ideal(x1*x2*x3*x4,x2*x3*x4*x5,x1*x3*x4*x5,x1*x2*x4*x5,
x1*x2*x3*x5)
zeroOneSOP(I)
\end{verbatim}
\end{procedure}

\begin{procedure}\label{procedure-NN}
Let $I\subset S$ be a monomial ideal. This procedure for \textit{Macaulay}$2$ \cite{mac2} compute 
a linear system of parameters and checks whether this is a regular
sequence on $S/I$. It determines whether the ring $S/I$ is
Cohen-Macaulay, and finds the chromatic and independence number of a
graph $G$.
It includes a dimension test for systems of parameters of $S/I$. This
procedure is used in Example~\ref{example-comp-Grotzsch-graph}.
\begin{verbatim}
restart
loadPackage "NoetherNormalization"
loadPackage "EdgeIdeals"
--The Mycielski--Grotzsch graph H has 11 vertices, 20 edges, 
--chromatic number 4, independence number 5. 
--The complement G of H has 11 vertices and 35 edges,
--independence number 2, and chromatic number 6
--I=I(G), J=I(H), edge ideals of G and H
S = QQ[x1,x2,x3,x4,x5,x6,x7,x8,x9,x10,x11];
I = ideal(x1*x3, x1*x4, x2*x4, x2*x5, x3*x5, x1*x6, x1*x8, 
x1*x9,x2*x7, x2*x9, x2*x10,x3*x6, x3*x8, x3*x10,x4*x6, 
x4*x7, x4*x9,x5*x7, x5*x8, x5*x10,x1*x11, x2*x11, x3*x11, 
x4*x11, x5*x11,x6*x7, x6*x8, x6*x9, x6*x10,x7*x8, x7*x9, 
x7*x10,x8*x9, x8*x10,x9*x10);
G=graph(I)
H=complementGraph(G)
J=edgeIdeal(H)
--This gives the dimension of S/I(H)
independenceNumber H, dim J
chromaticNumber(H)
--The first forms up to the dimension of S/I is 
--a Noether Normalization of S/I
toString noetherNormalization(S/I)
f1=(7/4)*x1+4*x2+(5/9)*x3+(1/7)*x4+(3/5)*x5+(7/6)*x6+
2*x7+(9/8)*x8+(5/6)*x9+x11
f2=(9/2)*x1+(2/7)*x2+(1/2)*x3+(1/5)*x4+9*x5+(6/7)*x6+
(3/8)*x7+(6/5)*x8+(10/9)*x9+x10
--This checks whether {f1,f2} is a system of parameters
dim(I+ideal(f1,f2))==0
--This checks whether {f1,f2} is a regular sequence on S/I(G)
(I:f1)==I
(I+ideal(f1)):f2==I+ideal(f1)
--If the following equality is 0, 
--then S/I(G) is Cohen-Macaulay
pdim coker gens gb I==codim I
\end{verbatim}
\end{procedure}

\end{appendix}

\section*{Acknowledgments.} 
The software system \textit{Macaulay}$2$ \cite{mac2}, and the packages 
\textit{NoetherNormalization} \cite{noethernormalization},
\textit{EdgeIdeals} \cite{edgeideals},  
were used to implement algorithms for computing
0-1 systems of parameters, Noether normalizations of
monomial ideals, and graph invariants. The  computer
algebra system \textit{MAGMA} \cite{magma} was used to compute basic
parameters of linear codes.

\section*{Statements and Declarations}  
On behalf of all authors, the corresponding author states that there is no conflict of interest.

No funding was received for conducting this study.

The authors have no relevant financial or non-financial interests to disclose.

Data sharing is not applicable to this article as no datasets were
generated or analyzed during the current study.  

\bibliographystyle{plain}

\begin{thebibliography}{10}

\bibitem{Adams-Reiner} A. Adams and V. Reiner, A colorful Hochster
formula and universal parameters for face rings, \textit{J. Commut.
Algebra}
{\bf 15} (2023), no.~2, 151--176. 

\bibitem{Ball1} S. Ball, On sets of vectors of a finite vector space in which every
subset of basis size is a basis, \textit{J. Eur. Math. Soc. (JEMS)} {\bf 14}
(2012), no.~3, 733--748.

\bibitem{Ball2} S. Ball and J. De~Beule, On sets of vectors of a
finite vector space in which every subset of basis size is a basis
II, \textit{Des. Codes Cryptogr.} {\bf 65} (2012), no.~1-2, 5--14. 

\bibitem{Ball3} S. Ball, Extending small arcs to large arcs, \textit{Eur. J.
Math.} {\bf 4} (2018), no.~1, 8--25.

\bibitem{magma} W. Bosma, J. Cannon and C. Playoust, The Magma 
algebra system. I. The user language, \textit{J. Symbolic Comput.} {\bf 24}
(1997), 235--265.  

\bibitem{BHer}{W. Bruns and J. Herzog, {\em Cohen-Macaulay Rings\/},
Cambridge University Press, Cambridge, Revised Edition, 1998.}   

\bibitem{chvatal} V. Chv\'atal, On certain polytopes associated with
graphs,  {\it J. Combin. Theory Ser. B} {\bf 18} (1975), 138--154. 

\bibitem{dougherty} S.~T. Dougherty, {\it Combinatorics and finite geometry}, 
Springer Undergraduate Mathematics Series, Springer, Cham, 2020.

\bibitem{Eisen}{D. Eisenbud, {\it Commutative Algebra with a View
toward Algebraic Geometry\/}, Graduate
Texts in  Mathematics {\bf 150}, Springer-Verlag, 1995.}

\bibitem{FHV}{C. Francisco, H.T. H$\rm \grave{a}$ and A. Van Tuyl, Colorings of
  hypergraphs, perfect graphs, and associated primes of powers of
  monomial ideals, \textit{J. Algebra} {\bf 331} (2011), no. 1,
224--242.}


\bibitem{edgeideals} C. Francisco, A. Hoefel and A. Van Tuyl, EdgeIdeals: a package for
(hyper)graphs, {\it J. Softw. Algebra Geom.} {\bf 1} (2009), 1--4.


\bibitem{Gentle} J.~E. Gentle, {\it Matrix algebra---theory,
computations and applications in statistics}, third edition, 
Springer Texts in Statistics, Springer, Cham, 2024. 

\bibitem{Ghorpade-Lachaud} S. Ghorpade and G. Lachaud, 
Hyperplane sections of Grassmannians and the number of MDS linear
codes, Finite Fields Appl. {\bf 7} (2001), no. 4, 468--506.

\bibitem{clutters}{I. Gitler, C. Valencia and R. H. Villarreal, 
A note on Rees algebras and the MFMC property, {\it Beitr\"age
Algebra Geom.} {\bf 48}
(2007), no. 1, 141--150.}

\bibitem{mac2}{D. Grayson and M. Stillman, {\em Macaulay\/}$2$, 1996.
Available at \url{https://www.macaulay2.com/}}   

\bibitem{Ha-VanTuyl} H. T. H$\rm \grave{a}$ and A. Van Tuyl, Monomial ideals,
edge ideals of hypergraphs, and their graded Betti numbers, {\it J.
Algebraic Combin.} {\bf 27} (2008), 215--245.

\bibitem{herzog-hibi-book} J. Herzog and T. Hibi, {\it Monomial
Ideals}, 
Graduate
Texts in  Mathematics {\bf 260}, Springer-Verlag, 2011.

\bibitem{2021-HerzogMoradi} J. Herzog and S. Moradi, 
Systems of parameters and the Cohen-Macaulay property,  
J. Algebraic Combin. {\bf 54} (2021), no. 4, 1261--1277. 

\bibitem{Huffman-Pless} W. C. Huffman and V. Pless,
\textit{Fundamentals of error-correcting 
codes}, Cambridge University Press, Cambridge, 2003. 

\bibitem{kind} B. Kind and P. Kleinschmidt, Sch\"albare Cohen-Macauley-Komplexe und
ihre Parametrisierung, \textit{Math. Z.} {\bf 167} (1979), no.~2, 173--179.

\bibitem{lovasz} L. Lov\'asz, Normal hypergraphs and the perfect
graph conjecture, {\it Discrete Math.} {\bf 2} (1972), no. 3, 253--267. 

\bibitem{ass-powers} J. Mart\'{\i}nez-Bernal, 
S. Morey, and R. H.
Villarreal, Associated primes of powers of edge ideals, 
{\it Collect. Math.} {\bf 63} (2012),
no. 3, 361--374.

\bibitem{Mats}{H. Matsumura, {\it Commutative Ring Theory\/}, 
Cambridge
Studies in Advanced Mathematics {\bf 8}, 
Cambridge University Press, 1986.}

\bibitem{readdy} M.~A. Readdy, The Yuri Manin ring and its $\mathscr{B}_n$-analogue, 
\textit{Adv. in Appl. Math.} {\bf 26} (2001), no.~2, 154--167.

\bibitem{Santos} D. A. Santos, Linear Algebra Notes. Revision, 
January 2, 2010.\\
\url{https://scipp.ucsc.edu/~haber/ph116A/new_linearalgebra.pdf}

\bibitem{Smith} D.~E. Smith, On the Cohen-Macaulay property in
commutative algebra and simplicial topology, Pacific J. Math. {\bf
141} (1990), no.~1, 165--196.

\bibitem{noethernormalization}
B. Snapp and N. Stapleton, {\em NoetherNormalization: place an ideal in
Noether normal position}, \textit{Macaulay}$2$ package. Available at
\url{https://github.com/Macaulay2/M2/tree/master/M2/Macaulay2/packages}

\bibitem{soifer} A. Soifer, {\it The new mathematical coloring book---
mathematics of coloring and the colorful life of its creators},
second edition,  
Springer, New York, 2024.

\bibitem{Sta1}{R. P. Stanley, Hilbert functions of graded algebras,
Adv. Math. {\bf 28} (1978), 57--83.}  

\bibitem{Stanley-sop}{R. Stanley, Balanced {C}ohen-{M}acaulay 
complexes, Trans. Amer. Math. Soc. {\bf 249} (1979), 
139--157.}

\bibitem{Sta2}{R. P. Stanley, {\it Combinatorics and Commutative
Algebra\/}, Birkh{\"a}user Boston, 2nd ed., 1996.}   

\bibitem{Vi2}{R. H. Villarreal, Cohen--{M}acaulay graphs, {\it Manuscripta
Math.} {\bf 66} (1990), 277--293.}

\bibitem{monalg-3rd-edition} R. H. Villarreal, \textit{Monomial
Algebras}, third edition, Monographs and Research Notes in Mathematics,
Chapman and Hall/CRC, Boca Raton, FL, 2026.  
\end{thebibliography}

\end{document}